\documentclass{amsart}
\usepackage{amsmath}
\usepackage{amssymb}
\usepackage[all]{xy}
\usepackage{graphicx}
\usepackage{bbm}
\usepackage{commath}
\usepackage[numbers]{natbib}
\usepackage{caption}
\usepackage{appendix}
\usepackage{enumerate}
\usepackage{amsmath,leftidx}
\usepackage{constants}
\usepackage[margin=3cm]{geometry}
\usepackage[margin=3cm]{geometry}
\usepackage{xcolor}
\usepackage{caption}

\numberwithin{equation}{section}

\usepackage{enumitem}

\usepackage{url}
\usepackage{longtable,tabu}
\usepackage{float}

\usepackage{xspace,amssymb,amsfonts,euscript}
\usepackage{amsthm,amsmath}
\usepackage{palatino}
\usepackage{euscript}
\input xy \xyoption {all}

\RequirePackage{ifpdf}
\ifpdf
  \IfFileExists{pdfsync.sty}{\RequirePackage{pdfsync}}{}
  \RequirePackage[pdftex,
   colorlinks = true,
   pagebackref,
   bookmarksopen=true]{hyperref}
\else
  \RequirePackage[hypertex]{hyperref}
\fi

\RequirePackage{ae, aecompl, aeguill} 

\newtheorem{thm}{Theorem}[section]
\newtheorem{lem}[thm]{Lemma}
\newtheorem{prop}[thm]{Proposition}

\renewcommand{\tilde}{\widetilde}

\theoremstyle{remark}
\newtheorem{remark}[thm]{Remark}

\theoremstyle{definition}
\newtheorem{definition}[thm]{Definition}

\makeatletter
\newcommand\leftdash{\!\rotatebox[origin=c]{-60}{$\dabar@\dabar@\dabar@$}\!}
\newcommand\rightdash{\!\rotatebox[origin=c]{60}{$\dabar@\dabar@\dabar@$}\!}
\makeatother

\makeatletter
\@date
\let\@@pmod\pmod
\DeclareRobustCommand{\pmod}{\@ifstar\@pmods\@@pmod}
\def\@pmods#1{\mkern4mu({\operator@font mod}\mkern 6mu#1)}
\makeatother

\newcommand{\Q}{\mathbb{Q}}

\newcommand{\Z}{\mathbb{Z}}

\newcommand{\one}{\mathbf{1}}

\newcommand{\re}{\mathrm{Re}}

\newcommand{\im}{\mathrm{Im}}
\renewcommand{\Im}{\mathrm{Im}}
\newcommand{\GL}{\mathrm{GL}}

\newconstantfamily{abcon}{symbol=c}

\DeclareMathOperator{\Sym}{\textup{Sym}}

\DeclareMathOperator{\Li}{\textup{Li}}

\newcommand{\abold}{\mathfrak{a}}

\newcommand{\mbold}{\mathfrak{m}}
\newcommand{\nbold}{\mathfrak{n}}

\hypersetup{
  colorlinks,
  citecolor=blue,
  linkcolor=red,
  urlcolor=blue}

\usepackage{enumerate}

\graphicspath{ {./images/} }

\title[]{Effective Sato--Tate distributions for Surfaces Arising from Products of Elliptic Curves}
\author{Quanlin Chen}
\address{Department of Mathematics, Stanford University}
\email{quanlinc@stanford.edu}

\author{Eric Shen}
\address{Department of Mathematics, Harvard University}
\email{ericshen1@college.harvard.edu}

\begin{document}

\maketitle

\begin{abstract}
We prove, with an unconditional effective error bound, the Sato--Tate distributions for two families of surfaces arising from products of elliptic curves, namely a one-parameter family of K3 surfaces and double quadric surfaces. To prove these effective Sato-Tate distributions, we prove an effective form of the joint Sato--Tate distribution for two twist-inequivalent elliptic curves, along with an effective form of the Sato--Tate distribution for an elliptic curve for primes in arithmetic progressions. The former completes the previous work \cite{thorner2021effective} of Thorner by including the cases in which one of the elliptic curves has CM. 
\end{abstract}

\section{Introduction and Statement of Result}\label{sintro}

Let $E/\Q$ be an elliptic curve over $\Q$ of conductor $N_E$ without complex multiplication (CM).  Hasse proved that for each prime $p$, the group $E(\mathbb{F}_p)$ of $\mathbb{F}_p$-rational points on the reduction of $E$ modulo $p$ satisfies the bound
\[
|p+1-\# E(\mathbb{F}_p)|<2\sqrt{p}.
\]
For $p\nmid N_E$, we define $a_E(p):=p+1-\#E(\mathbb{F}_p)$.  Hasse's bound implies that there exists $\theta_E(p)\in[0,\pi]$ such that
\[
a_E^*(p):=\frac{a_E(p)}{\sqrt{p}}=2\cos\theta_{E}(p). 
\]
In 2011, Barnet-Lamb, Geraghty, Harris, and Taylor \citep{barnet2011family} proved the celebrated Sato--Tate conjecture, which states that for a fixed subinterval $[a,b]\subseteq [-2,2]$, we have that
\[
\lim_{x\to \infty}\frac{\# \left\{ p \leq x: a_{E}^{*}(p) \in [a,b]\right\}}{\# \{p \leq x\}} = \frac{1}{\pi} \int_{a}^{b} \sqrt{1 - \Big ( \frac{t}{2} \Big )^2} dt.
\]
Thorner \citep{thorner2021effective} quantified the rate of convergence with effective dependence on $E$.  In particular,  for every subinterval $[a,b]\subset [-2,2]$, we have that\footnote{In this paper, all implied constants are positive, absolute, and effectively computable.  The sequence $c_1,c_2,c_3,\ldots$ denotes a sequence of certain positive, absolute, and effectively computable constants.}
\[
\bigg|\frac{\# \left\{ p \leq x: a_{E}^{*}(p) \in [a,b]\right\}}{\# \{p \leq x\}} - \frac{1}{\pi} \int_{a}^{b} \sqrt{1 - \Big ( \frac{t}{2} \Big )^2} dt\bigg|\ll \frac{\log(N_E\log x)}{\sqrt{\log x}},\qquad x\geq 3.
\]
The proof crucially relies on the work of Newton and Thorne \citep{newton2021symmetric,thorne2021symmetric}, proving that for all integers $m\geq 1$, the $m$-th symmetric power $L$-function $L(s,\mathrm{Sym}^m E)$ is the $L$-function of a unitary cuspidal automorphic representation of $\mathrm{GL}_{m+1}$.  If for each $m\geq1$ the generalized Riemann hypothesis is known for $L(s,\mathrm{Sym}^m E)$, then one can prove a more rapid rate of convergence.

For a closed subinterval $I\subseteq[0,\pi]$, a change of variables yields the equivalent statement
\[
\bigg|\frac{\# \left\{ p \leq x: \theta_{E}(p) \in I\right\}}{\# \{p \leq x\}} - \frac{2}{\pi} \int_{I} (\sin \theta)^2 d\theta\bigg|\le \Cl[abcon]{satotatestatement}\frac{\log(N_E\log x)}{\sqrt{\log x}},\qquad x\geq 3.
\]
Note that $\frac{2}{\pi}(\sin\theta)^2 d\theta$ is the pushforward of the Haar measure on $\text{SU}_2(\mathbb{C})$ under the trace map.  This observation leads to a natural generalization of the Sato--Tate conjecture to other abelian varieties $A/\mathbb{Q}$.  One might hope to find a suitable topological group $\mathrm{ST}(A)$ (the Sato--Tate group of $A$) such that the sequence $(x_p)$ of conjugacy classes of normalized images of Frobenius elements in $\mathrm{ST}(A)$ at primes $p$ of good reduction are equidistributed with respect to the pushforward of the Haar measure on $\mathrm{ST}(A)$ to its space of conjuacy classes.  See Sutherland \cite[Section 3]{sutherland2013sato} for further discussion, including a recipe for how one might construct a good candidate for $\mathrm{ST}(A)$.

In this paper, we provide the first unconditional effective bound of the error on the Sato--Tate distribution for two families of surfaces that arise from products of elliptic curves.
The first family of surfaces we consider is the one-parameter family of K3 surfaces first described in \citep{ahlgren2002zeta}, defined by
\begin{align*}
X_\lambda : s^2 = xy(x + 1)(y + 1)(x + \lambda y),
\end{align*}
where $\lambda\in \mathbb{Q} - \{0, -1\}$. It is not immediately evident that $X_\lambda$ is naturally related to a product of elliptic curves, but as detailed in \citep{ahlgren2002zeta}, the zeta function of $X_\lambda$ is related to the symmetric square of the zeta function of the Clausen elliptic curve $E_\lambda^{\text{Cl}}$, which is given by 
\begin{equation}
\label{eqn:Clausen_def}
    E_\lambda^{\text{Cl}}: y^2=(x-1)(x^2+\lambda).
\end{equation}

To more explicitly state the relation between the Clausen elliptic curve $E_\lambda^\text{Cl}$ and $X_\lambda$, we will introduce some notation. For a prime $p$, let $\phi_p$ be the unique quadratic character modulo $p$. For any two Dirichlet characters $\chi,\chi'$ over $(\mathbb{Z}/p\mathbb{Z})^{\times}$, define the normalized Jacobi sum by
$$
\binom{\chi}{\chi'}:=\frac{\chi'(-1)}{p}\sum_{x\in\mathbb{F}_p}\chi(x)\overline{\chi'}(1-x).
$$
Define 
$$
_3F_2(\lambda, p):=\frac{p}{p-1}\sum_{\chi (p)}\binom{\phi_{p}\chi}{\chi}^3\chi(\lambda),
$$
where the summation is taken over all Dirichlet characters modulo $p$. Note that this is a specific evaluation of the more general Gaussian hypergeometric function over finite fields, first defined by Greene in \cite{greene1984hypergeometric}, \cite{greene1987hypergeometric2}. Now, the Clausen elliptic curves $E_\lambda^{\textup{Cl}}$ and K3 surfaces $X_\lambda$ are related by
\begin{align}\label{k3eq1}
    |X_\lambda(\mathbb{F}_p)| = 1 + p^2 + 19p + p^2\leftidx{_3}{F}_2(-\lambda, p)
\end{align} and
\begin{align}\label{k3eq2}
    p + p^2\phi_p(\lambda + 1)_3F_2 \left ( \frac{\lambda}{\lambda + 1}, p \right ) = a_\lambda^\text{Cl}(p)^2,
\end{align}
where $a_\lambda^{\text{Cl}}(p)$ is the trace of Frobenius of  $E_\lambda^\text{Cl}$  \cite[Theorem 11.18]{ono2004web}, \cite[Proposition 4.1]{ahlgren2002zeta}. Now, the normalized trace of Frobenius of $X_\lambda$ is
$$
a_{X_\lambda}^{*}(p)=p\leftidx{_3}{F}_2(-\lambda, p).
$$

The limiting distribution of $a_{X_\lambda}^{*}(p)$ varied over all $\lambda\in \mathbb{F}_p$ as $p\to \infty$ was studied by Ono, Saad, and Saikia in \citep{ono2021distribution} and \citep{saad2023distribution}. Let $B(x)$ be the ``Batman'' distribution defined over $[-3,3]$ by \begin{align*}
    B(x) &= \begin{cases}
                \frac{3 + x}{\sqrt{3 - 2x - x^2}} + \frac{3 - x}{\sqrt{3 + 2x - x^2}} & |x| < 1,\\
                \frac{3 - |x|}{\sqrt{3 + 2|x| - x^2}} & 1 \leq |x| \leq 3, \\
                0 & \text{otherwise.}
                \end{cases}
\end{align*} In \cite[Corollary 1.4]{saad2023distribution}, Saad proved that for any prime $p \geq 5$ and sub-interval $[a,b]\subset[-3,3]$, 
$$
\bigg|\frac{\#\{ \lambda\in\mathbb{F}_p: a_{X_\lambda}^{*}(p)\in [a,b]\}}{p}-\frac{1}{4\pi}\int_{a}^{b}B(t)dt\bigg| \leq \frac{110.84}{p^{1/4}}.
$$

In this paper, we prove the Sato--Tate distribution of $a_{X_\lambda}^{*}(p)$ for a fixed K3 surface $X_\lambda$ when $p$ varies, with an effective error bound. In the generic case that $\lambda + 1$ is not a rational square and $\lambda \not \in \{1/8, 1, -1/4, -1/64, -4, -64\}$, we recover the ``Batman'' distribution in \citep{ono2021distribution}. 

\begin{thm}\label{k3case1}

Let $\lambda\in\mathbb{Q}$ satisfy
\begin{align*}
\lambda\notin\{r\in\mathbb{Q}\colon \sqrt{r+1}\in\mathbb{Q}\}\cup\{-64,-4,-\tfrac{1}{4},-\tfrac{1}{64},\tfrac{1}{8},1\}.
\end{align*}
Let $\lambda_1,\lambda_2\in\mathbb{Z}$ satisfy $\gcd(\lambda_1,\lambda_2)=1$ and $\lambda +1 = \frac{\lambda_1}{\lambda_2}$.  Let $q_{\lambda}$ be the squarefree part of $\lambda_1 \lambda_2$.  Let $N_{\lambda}$ be the conductor of the Clausen elliptic curve $E_{-\lambda/(\lambda+1)}^{\mathrm{Cl}}$ defined by \eqref{eqn:Clausen_def}.  Then there exists an absolute constant $\Cl[abcon]{11}$ such that if $x\geq 3$ and $[a,b]\subseteq [-3,3]$, then
$$
\Big|\frac{\# \{ p \leq x: a_{X_\lambda}^{*}(p) \in [a,b]\}}{\# \{p \leq x\}} - \int_{a}^{b} B(t) dt\Big| \ll x^{-\Cr{11}/\sqrt{q_{\lambda}}}+\frac{\log(N_{\lambda} q_{\lambda}\log x)}{\sqrt{\log x}}.
$$
\end{thm}

\begin{thm} \label{k3case2}
Let $\lambda\in\{r\in\mathbb{Q}\colon \sqrt{r+1}\in\mathbb{Q}\}-\{0,-1,8\}$, and let $N_{\lambda}$ be the conductor of the Clausen elliptic curve $E_{-\lambda/(\lambda+1)}^{\mathrm{Cl}}$ defined by \eqref{eqn:Clausen_def}. If $x\geq 3$ and $[a,b]\subseteq[-3,3]$, then
$$
\Big|\frac{\# \{ p \leq x: a_{X_\lambda}^{*}(p) \in [a,b]\}}{\# \{p \leq x\}} - \int_{a}^{b} \frac{1}{2\pi}\sqrt{\frac{3-t}{1+t}} dt\Big| \ll \frac{\log(N_{\lambda} \log x)}{\sqrt{\log x}}.
$$
\end{thm}
\begin{remark}
    The ``Batman'' distribution and the distribution given by $\frac{1}{2\pi}\sqrt{\frac{3-t}{1+t}}$ are the pushforwards of the Haar measures on the corresponding Sato-Tate groups $\textup{O}_3(\mathbb{R})$ and $\textup{SO}_3(\mathbb{R})$ under the trace map.
\end{remark}
\begin{remark}\label{rmkk3}
The corresponding results for all other $\lambda\in\mathbb{Q}-\{0,-1\}$ are given in Theorem \ref{k3main}.
\end{remark}

Now, consider any two twist-inequivalent elliptic curves
\[
E: y^2z=x^3+\Lambda_1xz^2+\Lambda_2z^3
\]
and
\[
E': y^2z'=x'^3+\Lambda_3x'z'^2+\Lambda_4z'^3
\]
over $\mathbb{Q}$. The second family of surfaces that we consider are the  \emph{double quadric surfaces}
\[
\mathcal{Z}=\mathcal{Z}(E,E'): y^2=zz'(x^3+\Lambda_1xz^2+\Lambda_2z^3)(x'^3+\Lambda_3x'z'^2+\Lambda_4z'^3).
\]
The normalized trace of Frobenius for $\mathcal{Z}$ is
\[
a_{\mathcal{Z}}^*(p)=\frac{1}{p}\sum_{x,x'}\left(\frac{x^3+\Lambda_2x+\Lambda_3}{p}\right)\left(\frac{x'^3+\Lambda_4x'+\Lambda_5}{p}\right)=a_{E}^{*}(p)a_{E'}^{*}(p).
\]
Our effective form of the Sato--Tate distribution for $\mathcal{Z}$ may be stated as follows.

\begin{thm}\label{doublequadriccase1}
Let $E$ and $E'$ be two twist-inequivalent non-CM elliptic curves with conductors $N_{E}$ and $N_{E'}$, respectively. Let $C_1(t)$ be given by 
$$
        C_1(t)=\frac{2}{\pi^2}\int_{|t|/2}^{2}\frac{1}{u}\sqrt{\bigg(1-\Big(\frac{u}{2}\Big)^2\bigg)\left(1-\Big(\frac{|t|}{2u}\Big)^2\right)}du.
$$ 
If $x \geq 16$ and $[a, b] \subset [-4, 4]$, then
$$
        \Big|\frac{\# \left\{ p \leq x: a_{\mathcal{Z}}^*(p) \in [a,b]\right\}}{\# \{p \leq x\}} - \int_a^b C_1(t) dt \Big| \ll {\frac{\sqrt{\log(N_{E}N_{E'}\log\log x)}}{\sqrt[4]{\log\log x}}}.
$$
\end{thm} 

\begin{remark}\label{rmkdoublequadric}
    
    The corresponding results in the cases where $E, E'$ are possibly CM are given in Theorem \ref{doublequadric}.
    
\end{remark}

In order to prove Theorems \ref{k3case1}, \ref{k3case2}, and \ref{doublequadriccase1}, we make effective the joint Sato--Tate distribution for any pair of elliptic curves $E$ and $E'$.  While generically $E$ will not be a twist of $E'$, the effective joint Sato--Tate distribution for $E$ and $E'$ when they are twist-equivalent can be recovered by understanding the effective Sato--Tate distribution for $E$ with the primes restricted to certain arithmetic progressions.  With this reduction in mind, we state our main result as being a classification of the effective joint Sato--Tate distributions for arbitrary pairs of elliptic curves.  We will deduce Theorems \ref{k3case1}, \ref{k3case2}, and \ref{doublequadriccase1} as corollaries. 

To state our main result, we first introduce some notation. Given coprime integers $a, q \geq 1$, define
\begin{align*}
    \pi(x) &:= \# \{p \leq x\} \\
    \pi(x; q, a) &:= \# \{p \leq x : p \equiv a \bmod q\}.
\end{align*}
Now, let $E$ and $E'$ be twist-inequivalent elliptic curves of conductors $N_E$ and $N_{E'}$, respectively. Given any interval $I$ and real number $r$, let $\mathbf{1}_{I}$ be the indicator function for $I$ and let $\mathbf{1}_{r\in I}:=\mathbf{1}_I(r)$. Assuming $I,I'\subseteq[0,\pi]$, define
\[
\pi_{E,I}(x):=\sum_{p\le x}\mathbf{1}_{I}(\theta_E(p)),
\]
\[
\pi_{E,I}(x; q,a):=\sum_{\substack{p\le x\\ p\equiv a \bmod{q}}}\mathbf{1}_{I}(\theta_E(p)),
\]
\[
\pi_{E,E',I,I'}(x):=\sum_{p\le x}\mathbf{1}_{I}(\theta_E(p))\mathbf{1}_{I'}(\theta_{E'}(p)).
\]
Define
\[
\mu_{\mathrm{CM}}(I):=\frac{|I|}{2\pi}+\frac{1}{2}\mathbf{1}_{\frac{\pi}{2}\in I},\qquad \mu_{\mathrm{ST}}(I)=\frac{2}{\pi}\int_{I}(\sin\theta)^2 d\theta.
\]
Finally, if $E$ has complex multiplication (CM) over an imaginary quadratc field $K$, then let $D_K$ be the discriminant of $K$ and define
\[
\delta(K,q)=\begin{cases}
1&\mbox{if $D_K | \thinspace q$,}\\
0&\mbox{otherwise.}
\end{cases}
\]
We may now state our main result as follows.

\begin{thm}\label{master} Let $E, E'$ be two twist-inequivalent elliptic curves over $\mathbb{Q}$. Let $a, q$ be coprime positive integers, let $I = [\alpha, \beta] \subseteq [0, \pi]$, and let $I' = [\alpha', \beta'] \subseteq [0, \pi]$. Finally, if $E, E'$ both have complex multiplication over a (possibly distinct) imaginary quadratic field, denote those two fields as $K$ and $K'$, respectively.
\begin{enumerate}
    \item The following are true for all $x \geq 16$.
    \begin{enumerate}
        \item If $E$ and $E'$ are both non-CM,
        \begin{align*}
                    \Big|\pi_{E, E', I, I'}(x)- \mu_{\textup{ST}}(I)\mu_{\textup{ST}}(I')\pi(x)\Big|\ll \pi(x){\frac{\log(N_{E}N_{E'}\log\log x)}{\sqrt{\log\log x}}}.
                    \end{align*}
        \item If $E$ is CM but $E'$ is non-CM,  \begin{align*}
                    \Big|\pi_{E, E', I, I'}(x)-\mu_{\mathrm{CM}}(I)\mu_{\textup{ST}}(I')\pi(x)\Big|\ll \pi(x)\frac{\log(N_{E})^4\log(N_{E'}\log x)}{\sqrt{\log x}}.
                    \end{align*}
        \item If $E, E'$ are both CM:  
            \begin{enumerate}
                \item When the discriminants of $K$ and $K'$ are coprime, there exists an absolute constant $\Cl[abcon]{161c1}$ such that \begin{align*}
                    \Big|\pi_{E, E', I, I'}(x)-\mu_{\mathrm{CM}}(I)\mu_{\mathrm{CM}}(I')\pi(x)\Big|\ll \pi(x)\exp\Big(\frac{-\Cr{161c1}\log x}{\sqrt{\log x} +3\log N_{E}N_{E'}}\Big).
                    \end{align*}
                \item When the discriminants of $K$ and $K'$ are not coprime, there exists an absolute constant $\Cl[abcon]{161c2}$ such that \begin{align*}
            \Big|\pi_{E, E', I, I'}(x)-\frac{1}{2}\Big(\frac{|I||I'|}{\pi^2}+\mathbf{1}_{\frac{\pi}{2}\in I}\mathbf{1}_{\frac{\pi}{2}\in I'}\Big)\pi(x)\Big|\ll \pi(x)\exp\Big(\frac{-\Cr{161c2}\log x}{\sqrt{\log x}+3\log N_{E}N_{E'}}\Big).
        \end{align*}
            \end{enumerate}
    \end{enumerate}
    \item The following are true for all $x \geq 16$.
    \begin{enumerate}
        \item If $E$ is non-CM, then there exists an absolute constant $\Cl[abcon]{162b1}$ such that
        $$
        \Big|\pi_{E, I}(x; q,a)-\mu_{\textup{ST}}(I)\frac{\pi(x)}{\varphi(q)}\bigg|\ll \frac{\pi(x)}{\varphi(q)}\Big(x^{-\Cr{162b1}/\sqrt{q}}+\frac{\log(N_Eq\log x)}{\sqrt{\log x}}\bigg).
        $$
        
        \item If $E$ is CM, then there exists an absolute constant $\Cl[abcon]{162b2}$ such that
        \begin{align*}
        \Big|\pi_{E, I}(x; q,a)&-\mu_{\textup{CM}}(I)\frac{\pi(x)}{\varphi(q)}- \chi_K(a)\delta(K,q)\frac{\pi(x)}{2\varphi(q)}\Big(\frac{|I|}{\pi}-\one_{\pi/2 \in I}\Big)\Big| \\&\ll \frac{\pi(x)}{\varphi(q)}\Big(x^{-\Cr{162b1} /\sqrt{q}} + (\log x)^{9/2} \exp \Big ( \frac{-\Cr{162b2}\log x}{\sqrt{\log x} + \log(N_Eq)} \Big ) \Big).
\end{align*}
    \end{enumerate}
\end{enumerate}
\end{thm}

This paper is organized as follows. In Section \ref{varieties} we state and prove the general results for the effective Sato-Tate distributions for K3 surfaces and double quadric surfaces, including Theorem \ref{k3case1}, Theorem \ref{k3case2} and Theorem $\ref{doublequadriccase1}$. The proof applies Theorem \ref{master} which will be proven later. In Section \ref{sinfo}, we present background information on Beurling-Selberg polynomials, elliptic curves, and symmetric power $L$-functions that will be required to prove Theorem \ref{master}. In Sections \ref{szerofree} and \ref{sPNT}, we prove all the zero-free regions and prime number theorems (respectively), used in the proof of Theorem \ref{master}. In Sections \ref{2a}, \ref{2b}, \ref{1b}, \ref{1c}, we prove parts 2.a, 2.b, 1.b, and 1.c of Theorem \ref{master}, respectively. In Appendix \ref{figure}, we plot the Sato--Tate distributions for all surfaces covered in Theorem \ref{k3main} and Theorem \ref{doublequadric} against numerical examples.

\section*{Acknowledgements}

The authors would like to thank Jesse Thorner for advising this project and for many helpful comments and discussions, and Ken Ono and Hasan Saad for many helpful discussions and valuable comments. The authors were participants in the 2022 UVA REU in Number Theory. They are grateful for the support of grants from the National Science Foundation (DMS-2002265, DMS-2055118, DMS-2147273), the National Security Agency (H98230-22-1-0020), and the Templeton World Charity Foundation. The authors used Wolfram Mathematica for computations.

\section{Proofs of Theorem \ref{k3case1}, Theorem \ref{k3case2}, and Theorem \ref{doublequadriccase1}}\label{varieties}

\subsection{Statement of the general results}

We will start by stating the most general results for the effective Sato-Tate distributions of K3 surfaces and double quadric surfaces. Throughout this section, let $\lambda\in \mathbb{Q}-\{0,1\}$ denote a rational number. Let $\lambda_1,\lambda_2\in\mathbb{Z}$ satisfy $\gcd(\lambda_1,\lambda_2)=1$ and $\lambda +1 = \frac{\lambda_1}{\lambda_2}$, and let $q_{\lambda}$ be the squarefree part of $\lambda_1 \lambda_2$.  Let $N_{\lambda}$ be the conductor of the Clausen elliptic curve $E_{-\lambda/(\lambda+1)}^{\mathrm{Cl}}$ defined in \eqref{eqn:Clausen_def}. For a real $r \in \mathbb{R}$ and closed interval $I=[a,b]\subset \mathbb{R}$, define $\mathbf{1}_{r\in I}$ to be $1$ if $r\in I$ and 0 otherwise. Let $E$ and $E'$ be two elliptic curves of conductors $N_E$ and $N_{E'}$, respectively. Finally, define the ``flying Batman'' distribution $B_1(x)$ by \begin{align*}
    B_1(x) &= \begin{cases}
                \frac{1}{4\pi\sqrt{3 - 2x - x^2}} + \frac{1}{4\pi\sqrt{3 + 2x - x^2}} & |x| < 1,\\
                \frac{1}{4\pi\sqrt{3 + 2|x| - x^2}} & 1 \leq |x| \leq 3, \\
                0 & \text{otherwise.}
                \end{cases}
\end{align*}

\begin{thm}\label{k3main}
Fix the notation in Section \ref{sintro}. Then there exists absolute constants $\Cr{162b1}, \Cr{162b2}$ such that the following are true for all $x \geq 16$.

\begin{enumerate}
    \item If $\lambda \not \in (\mathbb{Q}^2-1) \cup \{1/8, 1, -1/4, -1/64, -4, -64\}$, then for every subinterval $[a,b]\subset[-3,3]$, 
    \begin{align*}
        \frac{\# \left\{ p \leq X: a_{X_\lambda}^{*}(p)\in [a,b]\right\}}{\# \{p \leq X\}} - \int_{a}^{b} B(t) dt &\ll x^{-\frac{\Cr{162b1}}{\sqrt{q}}}+\frac{\log(N_Eq\log x)}{\sqrt{\log x}}.
    \end{align*}
    \item  If $\lambda \in (\mathbb{Q}^2-1)-\{0,-1,8\}$, then for every subinterval $[a,b]\subset[-1,3]$, 
    \begin{align*}
        \Big|\frac{\# \left\{ p \leq x: a_{X_\lambda}^{*}(p) \in [a,b]\right\}}{\# \{p \leq x\}} - \int_{a}^{b} \frac{1}{2\pi}\sqrt{\frac{3-t}{1+t}} dt\Big| \ll \frac{\log(N_E \log x)}{\sqrt{\log x}}.
    \end{align*}
    \item If $\lambda \in \{1/8, 1, -1/4, -1/64\}$, then for every subinterval $[a,b]\subset[-3,3]$, 
    \begin{align*}
        &\Big|\frac{\# \left\{ p \leq x: a_{X_\lambda}^{*}(p) \in [a,b]\right\}}{\# \{p \leq x\}} - \frac{\one_{1 \in [a, b]}+\one_{-1 \in [a, b]}}{4}- \int_{a}^{b} B_1(t) dt\Big|
        \\&\ll x^{-\Cr{162b1} /\sqrt{q}} + (\log x)^{9/2} \exp \Big ( \frac{-\Cr{162b2}\log x}{\sqrt{\log x} + \log(N_Eq)} \Big ).
    \end{align*}
    \item If $\lambda \in \{-4, -64\}$, then for every subinterval $[a,b]\subset[-1,3]$,
    \begin{align*}
        &\Big|\frac{\# \left\{ p \leq x: a_{X_\lambda}^{*}(p) \in [a,b]\right\}}{\# \{p \leq x\}} -\frac{\one_{1 \in [a, b]}}{2} - \int_{a}^{b} \frac{dt}{2\pi\sqrt{3 + 2t - t^2}} \Big| 
        \\&\ll x^{-\Cr{162b1} /\sqrt{q}} + (\log x)^{9/2} \exp \Big ( \frac{-\Cr{162b2}\log x}{\sqrt{\log x} + \log(N_Eq)} \Big ).
    \end{align*}
    \item If $\lambda = 8$, then for every subinterval $[a,b]\subset[-1,3]$,
    \begin{align*}
        &\Big|\frac{\# \left\{ p \leq X: a_{X_\lambda}^{*}(p) \in [a,b]\right\}}{\# \{p \leq X\}} -\frac{\one_{-1 \in [a, b]}}{2} - \int_{a}^{b} \frac{dt}{2\pi\sqrt{3 + 2t - t^2}} \Big|
        \\&\ll x^{-\Cr{162b1} /\sqrt{q}} + (\log x)^{9/2} \exp \Big ( \frac{-\Cr{162b2}\log x}{\sqrt{\log x} + \log(N_Eq)} \Big ).
    \end{align*}
\end{enumerate}
\end{thm}

To state the Sato-Tate distributions for double quadric surfaces, we define the following functions
\begin{align*}
    C_1(t)&=\frac{2}{\pi^2}\int_{|t|/2}^{2}\frac{1}{u}\sqrt{\bigg(1-\Big(\frac{u}{2}\Big)^2\bigg)\left(1-\Big(\frac{|t|}{2u}\Big)^2\right)}du, \\
    C_2(t)&=\frac{1}{2\pi^2}\int_{|t|/2}^{2}\frac{1}{u}\sqrt{\frac{1-(u/2)^2}{1-(|t|/2u)^2}}du, \\
    C_3(t)&=\frac{1}{8\pi^2}\int_{|t|/2}^{2}\frac{1}{u\sqrt{(1-(u/2)^2)(1-(|t|/2u)^2)}}du.
\end{align*}

\begin{thm}\label{doublequadric}
Let $E$ and $E'$ be two twist-inequivalent elliptic curves of conductors $N_E$ and $N_{E'}$. The following are true for all $x \geq 16$.
\begin{enumerate}
        \item If $E$ and $E'$ are both non-CM, then
        $$
        \Big|\frac{\#\left\{ p \leq x: a_{\mathcal{Z}}^*(p) \in [a,b]\right\}}{\# \{p \leq x\}} - \int_a^b C_1(t) dt \Big| \ll {\frac{\sqrt{\log(N_{E}N_{E'}\log\log x)}}{\sqrt[4]{\log\log x}}}.
        $$
        \item If $E$ is non-CM and $E'$ is CM, then
        $$
        \Big|\frac{\# \left\{ p \leq x: a_{\mathcal{Z}}^*(p) \in [a,b]\right\}}{\# \{p \leq x\}} - \int_a^b C_2(t) dt - \frac{1}{2}\one_{0 \in I}\Big|\ll \frac{\sqrt[3]{\log(N_E)^4\log(N_{E'}\log x)}}{\sqrt[6]{\log x}}.
        $$
        \item If $E, E'$ are both CM, and the discriminants of their CM fields are coprime, then there exists an absolute constant $\Cl[abcon]{2quad1}$ such that 
        $$
        \Big|\frac{\# \left\{ p \leq x: a_{\mathcal{Z}}^*(p) \in [a,b]\right\}}{\# \{p \leq x\}} - \int_a^b C_3(t) dt - \frac{3}{4}\one_{0 \in I}\Big| \ll \exp\left(\frac{-\Cr{2quad1}\log x}{\sqrt{\log x} +3\log N_{E}N_{E'}}\right).
        $$
        \item If $E, E'$ are both CM elliptic curves, and the discriminants of their CM fields are not coprime, then there exists an absolute constant $\Cl[abcon]{2quad2}$ such that 
        $$\Big|\frac{\# \left\{ p \leq x: a_{\mathcal{Z}}^*(p) \in [a,b]\right\}}{\# \{p \leq x\}} - 2\int_a^b C_3(t) dt - \frac{1}{2}\one_{0 \in I}\Big| \ll\exp\left(\frac{-\Cr{2quad2}\log x}{\sqrt{\log x} +3\log N_{E}N_{E'}}\right).
        $$
\end{enumerate}
\end{thm} 

It is clear that Theorem \ref{k3case1} and \ref{k3case2} are special cases of Theorem \ref{k3main}, and Theorem \ref{doublequadriccase1} is a special case of Theorem \ref{doublequadric}. In the rest of this section, we will present proofs for Theorem \ref{k3main} and Theorem \ref{doublequadric} using Theorem \ref{master}, which will be proven in Sections \ref{2a} to \ref{1c}.

\subsection{Proof of Theorem \ref{k3main}}

For simplicity, we only present the proof of Theorem \ref{k3main} for the intervals $[a, b] \subset [0, 3]$; the other cases are proved in the same way. In this proof, we fix $\lambda$ and denote 
$$
E:=E_{-\lambda/(\lambda+1)}^\text{Cl}, \qquad a(p):=a_{E_{-\lambda/(\lambda+1)}^\text{Cl}}^*(p), \qquad q:=q_\lambda.
$$
Let $p$ be a prime that does not divide $q$ and let $\phi_p$ be the unique quadratic character modulo $p$.

\begin{proof}[Proof of Theorem \ref{k3main} (1)]

In this case, we have 
\begin{equation}\label{eqn:range_of_lambda_1}
\lambda\notin\{r\in\mathbb{Q}\colon \sqrt{r+1}\in\mathbb{Q}\}\cup\{-64,-4,-\tfrac{1}{4},-\tfrac{1}{64},\tfrac{1}{8},1\}.
\end{equation}
From \cite[Page 191]{ono2004web} we find that condition \eqref{eqn:range_of_lambda_1} implies $E$ is a non-CM elliptic curve. Note that $\lambda\notin\{r\in\mathbb{Q}\colon \sqrt{r+1}\in\mathbb{Q}\}$ implies $q > 1$. Applying the relations between $X_\lambda$ and its Clausen elliptic curve $E$ (\eqref{k3eq1} and \eqref{k3eq2}), we have 
$$
a_{X_\lambda}^{*}(p)=\phi_p(\lambda+1)(a(p)^2-1).
$$
Note that when $\phi_p(\lambda+1)=1$, $a_{X_\lambda}^{*}(p)\in [a, b]$ is equivalent to
\begin{align*}
     a(p) \in \Big [ \sqrt{1 + a}, \sqrt{1 + b}\Big ]
\end{align*} and when $\phi_p(\lambda+1)=-1$, $a_{X_\lambda}^{*}(p)\in [a, b]$ is equivalent to \begin{align*}
     a(p) \in \Big [ \sqrt{1 - b}, \sqrt{1 - a}\Big ],
\end{align*} where each endpoint is set to be $0$ if it is not real. By quadratic reciprocity, $\phi_p(\lambda+1)$ as a function of $p$ is periodic with period dividing $4q$ when $p>4q$. Since $q \neq 1$ and square-free, when $j$ spans $(\mathbb{Z}/T\mathbb{Z})^{\times}$, exactly half of the $\phi_j(q)$ will be equal to $1$ and the other half will equal $-1$. Apply Theorem \ref{master} (2.a) to modulus $4q$; with a simple calculation, we obtain
\begin{align}
    \nonumber \frac{\# \left\{ p \leq X: a_{X_\lambda}^{*}(p)\in [a,b]\right\}}{\# \{p \leq X\}}  &= \int_{a}^{b} B(t) dt + O\Big(x^{-\frac{\Cr{162b1}}{\sqrt{q}}}+\frac{\log(N_Eq\log x)}{\sqrt{\log x}}\Big).
\end{align} \end{proof}

\begin{proof}[Proof of Theorem \ref{k3main} (2)]
This proof differs from the proof of Theorem \ref{k3main} (1) only in that $\lambda+1$ is a rational square, and thus $\phi_p(q) = 1$ for all $p$ not dividing $q$. Hence we may apply \cite[Theorem 1.1]{thorner2021effective} to derive that for all $x \geq 3$,
$$
\nonumber \frac{\# \left\{ p \leq x: a_{X_\lambda}^{*}(p) \in [a,b]\right\}}{\#\{p \leq X\}} = \frac{1}{2\pi}\int_{a}^{b} \sqrt{\frac{3-t}{1+t}} dt + O\Big( \frac{\log(N_E\log x)}{\sqrt{\log x}}\Big).
$$ 
\end{proof}

\begin{proof}[Proof of Theorem \ref{k3main} (3, 4, 5)] 

By \cite[Page 191]{ono2004web}, when $\lambda \in \{8, 1/8, 1, -4, -1/4, -64, -1/64\}$, $E$ is a CM elliptic curve. Let $K$ be the CM field of $E$ and $D_{K}$ the discriminant of $K$. Also, let $T$ be the period of $\phi_p(\lambda + 1)$. All the possible values of $\lambda$ and the corresponding $T$ and $D_K$ are presented in Table \ref{table1}.

\begin{table}
\captionof{table}{Values of $T$ and $D_K$ for $\lambda$ satisfying that $E_{-\lambda/(1+\lambda)}^\text{Cl}$ has CM. \label{table1}}
\begin{tabular}{ | m{1.1cm} | m{1.1cm}| m{1.1cm}| m{1.1cm}| m{1.1cm}| m{1.1cm}| m{1.1cm}| m{1.1cm}| } 
\hline
\vspace{0.1 cm}$\lambda$\vspace{0.1 cm} & \vspace{0.1 cm}$8$ \vspace{0.1 cm}& \vspace{0.1 cm}$1/8$ \vspace{0.1 cm}& \vspace{0.1 cm}$1$ \vspace{0.1 cm}& \vspace{0.1 cm}$-4$ \vspace{0.1 cm}& \vspace{0.1 cm}$-1/4$ \vspace{0.1 cm}& \vspace{0.1 cm}$-64$ \vspace{0.1 cm}& \vspace{0.1 cm}$-1/64$\vspace{0.1 cm}\\ 
\hline
\vspace{0.1 cm}$T$\vspace{0.1 cm} & \vspace{0.1 cm}$1$\vspace{0.1 cm} & \vspace{0.1 cm}$8$\vspace{0.1 cm} & \vspace{0.1 cm}$8$\vspace{0.1 cm} & \vspace{0.1 cm}$6$\vspace{0.1 cm} & \vspace{0.1 cm}$12$\vspace{0.1 cm} & \vspace{0.1 cm}$14$\vspace{0.1 cm} & \vspace{0.1 cm}$28$\vspace{0.1 cm}\\ 
\hline
\vspace{0.1 cm}$D_K$\vspace{0.1 cm} & \vspace{0.1 cm}$-4$\vspace{0.1 cm} & \vspace{0.1 cm}$-4$\vspace{0.1 cm} & \vspace{0.1 cm}$-8$\vspace{0.1 cm} & \vspace{0.1 cm}$-3$\vspace{0.1 cm} & \vspace{0.1 cm}$-3$\vspace{0.1 cm} & \vspace{0.1 cm}$-7$\vspace{0.1 cm} & \vspace{0.1 cm}$-7$\vspace{0.1 cm}\\ 
\hline
\end{tabular}
\end{table}

The proof now proceeds in the same way as that of Theorem \ref{k3main} (1). The only difference is that we apply Theorem \ref{master} (2.b) to modulus $T$, instead of (2.a) as in the above two cases. \end{proof}

\subsection{Proof of Theorem \ref{doublequadric}}

Let $E$ and $E'$ be two twist-inequivalent elliptic curves. For all subintervals $I \subset [-2, 2]$, define the semicircular measure $S(I)$ as$$
S(I)=\int_I \frac{1}{\pi}\sqrt{1-(t/2)^2}dt,
$$
and the reciprocal $T(I)$ of the semicircular distribution as
$$
T(I)=\int_I \frac{1}{2\pi\sqrt{1-(t/2)^2}}dt.
$$
Consider 
$$
R=\{(x,y)\in [-2,2]\times[-2,2]: a\le xy\le b\}.
$$
For an odd integer $L > 0$ that will be specified differently in each case, consider the $2L-2$ lines $x = -2 + 4k/L$ and $y = -2 + 4k/L$ for $1\le k \le L-1$. They divide the cell $[-2,2]\times [-2,2]$ into the $L^2$ smaller cells 
$$
\left[-2 + \frac{4i-4}{L}, -2 + \frac{4i}{L}\right]\times \left[-2 + \frac{4i'-4}{L}, -2 + \frac{4i'}{L}\right], \qquad 1\le i, j\le L.
$$
Since $a_E^*(p)$ and $a_{E'}^*(p)$ are both irrational whenever they are nonzero, no point $(a_E^*(p), a_{E'}^*(p))$ lies on the boundary of any of the above $L^2$ cells. 

\begin{proof}[Proof of Theorem \ref{doublequadric} (1)]

Consider any $1 \leq i, i' \leq L$, and define $I=[-2 + (4i-4)/L, -2 + 4i/L]$ and $I'=[-2 + (4j-4)/L, -2 + 4j/L]$. Applying Theorem \ref{master} (1.a) to $I$ and $I'$, we have \begin{align}
    \nonumber\#\{p\le x: (a_E^*(p),a_{E'}^*(p))\in I\times I'\}&=S(I)S(I')\pi(x)+O\bigg(\pi(x){\frac{\log(N_{E}N_{E'}\log\log x)}{\sqrt{\log\log x}}}\bigg)\\&\label{eqn:proof_doublequadric_1}\ll \pi(x)\bigg(\frac{1}{L^2}+{\frac{\log(N_{E}N_{E'}\log\log x)}{\sqrt{\log\log x}}}\bigg).
\end{align}

Let $\mathcal{S}$ be the union of all the smaller cells whose interior is contained in the interior of $R$. Let $\mathcal{T}$ be the union of all the smaller cells whose interior has non-empty intersection with $\{xy=a\}\cup\{xy=b\}$. It is easy to check that $\mathcal{T}$ contains $O(L)$ of the smaller cells. Now, we have
$$
\# \{p \leq x : (a_E^*(p), a_{E'}^*(p)) \in R\} = \# \{p \leq x : (a_E^*(p), a_{E'}^*(p)) \in \mathcal{S}\} + \#\{p\le x: (a_E^*(p),a_{E'}^*(p))\in \mathcal{T} \cap R\}.
$$

The $L-1$ vertical lines divide $\mathcal{S}$ into at most $L$ vertical strips each with width $1/L$. Applying Theorem \ref{master} (1.a) to those strips of $\mathcal{S}$, we obtain
\begin{align}\label{eqn:S_estimate}
    \#\{p\le x: (a_E^*(p),a_{E'}^*(p))\in \mathcal{S}\} & = \pi(x)\int_\mathcal{S}\frac{1}{\pi^2}\sqrt{\bigg ( 1 - \bigg (\frac{x}{2} \bigg )^2 \bigg )\bigg ( 1 - \bigg (\frac{y}{2} \bigg )^2 \bigg )}dxdy \\&\nonumber +O\bigg( \pi(x)\bigg(L{\frac{\log(N_{E}N_{E'}\log\log x)}{\sqrt{\log\log x}}}\bigg)\bigg).
\end{align} 
Note that
\begin{align}\label{eqn:negligble}
    \int_{R - \mathcal{S}} \frac{1}{\pi^2}\sqrt{\bigg ( 1 - \bigg (\frac{x}{2} \bigg )^2 \bigg )\bigg ( 1 - \bigg (\frac{y}{2} \bigg )^2 \bigg )}dxdy &\ll \frac{1}{L}
\end{align}
Applying \eqref{eqn:proof_doublequadric_1} to each cell of $\mathcal{T}$, we have 
\begin{equation}\label{eqn:T_estimate}
\#\{p\le x: (a_E^*(p),a_{E'}^*(p))\in \mathcal{T} \cap R\}\ll \pi(x)\bigg(L{\frac{\log(N_{E}N_{E'}\log\log x)}{\sqrt{\log\log x}}} + \frac{1}{L}\bigg).
\end{equation}
Collating \eqref{eqn:S_estimate}, \eqref{eqn:negligble}, and \eqref{eqn:T_estimate}, we obtain 
\begin{align*}
    \#\{p\le x: (a_E^*(p),a_{E'}^*(p))\in R\} & = \pi(x)\int_a^b C_1(t) +O\bigg( \pi(x)\bigg(L{\frac{\log(N_{E}N_{E'}\log\log x)}{\sqrt{\log\log x}}} + \frac{1}{L}\bigg)\bigg).
\end{align*} 
Taking $L=2\left \lfloor \frac{\sqrt[4]{\log\log x}}{\sqrt{\log(N_{E}N_{E'}\log\log x)}}\right \rfloor + 1$ concludes the proof. \end{proof}

\begin{proof}[Proof of Theorem \ref{doublequadric} (2)]

The proof proceeds similarly to that of Theorem \ref{doublequadric} (1). For simplicity, we only present the proof in the case that $0\notin [a,b]$; the full theorem may be proven in the same way. When $0\notin I'$, we have that by applying Theorem \ref{master} (1.b) and elementary calculus,
\begin{align*}
    \#\{p\le x: (a_E^*(p),a_{E'}^*(p))\in I\times I'\}&=S(I)T(I')\pi(x)+O\bigg(\pi(x)\frac{\log(N_{E})^4\log(N_{E'}\log x)}{\sqrt{\log x}}\bigg)\\&\ll \pi(x)\bigg(\frac{1}{L^{3/2}}+\pi(x)\frac{\log(N_{E})^4\log(N_{E'}\log x)}{\sqrt{\log x}}\bigg).
\end{align*}
Define $R$ similarly as in case (1). Then by a similar calculation as in case (1),
\begin{align*}
        \#\{p\le x: (a_E^*(p),a_{E'}^*(p))\in R\} =  \pi(x)\int_a^b C_2(t)+O\bigg(\pi(x)\bigg(L\frac{\log(N_{E})^4\log(N_{E'}\log x)}{\sqrt{\log x}} + \frac{1}{\sqrt{L}}\bigg)\bigg).
\end{align*}
Taking $L=2\left \lfloor\frac{\sqrt[3]{\log x}}{\sqrt[3]{\log(N_E)^8\log(N_{E'}\log x)^2}}\right \rfloor + 1$ now concludes the proof. 
\end{proof}

\begin{proof}[Proof of Theorem \ref{doublequadric} (3), (4)] The proofs of cases (3) and (4) proceed in exactly the same way, so we only present the proof for case (3). Define $\mathcal{T}$ and $R$ similarly as in the previous cases. Similar to case (2), we only consider the case when $0\notin [a,b]$. Now when $0\notin I$ and $0\notin I'$, we have 
\begin{align*}
    \#\{p\le x: (a_E^*(p),a_{E'}^*(p))\in I\times I'\}&=T(I)T(I')\pi(x)+O\bigg(\pi(x)\exp\bigg(\frac{-\Cr{161c1}\log x}{\sqrt{\log x} +3\log N_{E}N_{E'}}\bigg)\bigg)\\ &\ll \pi(x)\bigg(\frac{\min\{T(I),T(I')\}}{\sqrt{L}}+ \exp\bigg(\frac{-\Cr{161c1}\log x}{\sqrt{\log x} +3\log N_{E}N_{E'}}\bigg)\bigg).
\end{align*}
Now, consider any $c \in (-4, 4)$. Let $\mathcal{T}_c$ be the set of cells that have a nonempty intersection with the curve $\{xy = c\}$; evidently $|\mathcal{T}_c| \ll L(4 - |c|)$. By elementary calculus, for any $I \times I' \in S_c$,
\begin{align*}
    \min\{T(I),T(I')\} &\ll \frac{1}{L\sqrt{4 - \Big(\sqrt{c} + \frac{1}{L}\Big)^2}}.
\end{align*}
Thus 
$$
\sum_{I\times I'\in \mathcal{T}_c}\frac{\min\{T(I),T(I')\}}{\sqrt{L}}\ll \frac{1}{\sqrt{L}}.
$$ 
Hence, we have that
\begin{align*}
\#\{p\le x: (a_E^*(p),a_{E'}^*(p))\in \mathcal{T} \cap R\}\ll \pi(x)\bigg(L\exp\bigg(\frac{-\Cr{161c1}\log x}{\sqrt{\log x} +3\log N_{E}N_{E'}}\bigg)\bigg) + \frac{1}{\sqrt{L}}\bigg).
\end{align*}
Thus,
\begin{align*}\label{eqn:double_case_3}
    \#\{p\le x: (a_E^*(p),a_{E'}^*(p))\in R\} &= \pi(x) \bigg(\int_a^b C_3(t)+\frac{3}{4}\one_{0\in [a,b]} \bigg) 
    \\&+
    O\bigg(\pi(x)\bigg(L\exp\bigg(\frac{-\Cr{161c1}\log x}{\sqrt{\log x} +3\log N_{E}N_{E'}}\bigg)\bigg) + \frac{1}{\sqrt{L}}\bigg)\bigg).
\end{align*} 
Taking $L=2\left \lfloor\exp\bigg(\frac{2\Cr{161c1}\log x}{3\sqrt{\log x} +9\log N_{E}N_{E'}}\bigg)\right\rfloor + 1$ concludes the proof. \end{proof}

\section{Preliminaries for the Proof of Theorem \ref{master}}\label{sinfo}

In this section we introduce preliminary information that is needed for the proof of Theorem \ref{master}. Throughout this section, when $E$ (resp. $E'$) is an elliptic curve  with complex multiplication over an imaginary quadratic field $K$ (resp. $K'$), let $\chi_K$ (resp. $\chi_{K'}$) and $D_K$ (resp. $D_{K'}$) respectively denote the Kronecker character and the absolute discriminant of $K$ (resp. $K'$). Finally, for an interval $I$, let $\chi_I$ denote the characteristic function of $I$.

\subsection{Beurling-Selberg polynomials} 

As important technical tools, we will need both the original one-dimensional Beurling-Selberg polynomials, along with a two-dimensional analogue.  The following two Lemmas follow from \cite[Theorem 1]{montgomery_ten_lectures} and \cite[Theorem 2]{cochrane1988trigonometric}.

\begin{lem}\label{bspoly1}

Let $I \subseteq [0,\pi]$ be a closed subinterval and $M$ a positive integer. Then there exists an absolute constant $\Cl[abcon]{bspoly0} > 0$, and two polynomials 
$$
F^{\pm}_{I ,M}(\theta)=\sum_{0\le m\le M} \hat{F}_{I,M}^{\pm}(m)\cos(m\theta)
$$
such that
\begin{enumerate}
    \item for all $\theta\in [0,\pi]$
    $$
    F^{-}_{I,M}(\theta)\le \mathbf{1}_{I}(\theta)\le F^{+}_{I,M}(\theta);
    $$
    \item we have that
    $$
    \Big|\hat{F}_{I,M}^{\pm}(0)-\frac{|I|}{\pi}\Big| \leq \frac{\Cr{bspoly0}}{M};
    $$
    \item if $m\neq 0$, then
    $$
    \Big|\hat{F}_{I,M}^{\pm}(m)\Big| \leq \frac{\Cr{bspoly0}}{m}.
    $$
\end{enumerate}

\end{lem}

\begin{lem}\label{bspoly2}

Let $I, I' \subseteq [0,\pi]$ be two closed subintervals and $M$ a positive integer. Then there exists an absolute constant $\Cl[abcon]{bspoly00} > 0$, and two polynomials 
$$
F^{\pm}_{I,I',M}(\theta, \theta')=\sum_{0\le m, m'\le M} \hat{F}_{I,I',M}^{\pm}(m,m')\cos(m\theta)\cos(m'\theta')
$$
such that \begin{enumerate}
    \item for all $\theta,\theta'\in [0,\pi]$, 
    $$
    F^{-}_{I,I',M}(\theta,\theta')\le \chi_{I}(\theta)\chi_{I'}(\theta')\le F^{+}_{I,I',M}(\theta,\theta');
    $$
    \item we have that
    \[
    \Big|\hat{F}_{I,I',M}^{\pm}(0,0)-\frac{|I|}{\pi}\Big| \leq \frac{\Cr{bspoly00}}{M};
    \]
    \item if $m \neq 0$, then
    \[
    \Big|\hat{F}_{I,I',M}^{\pm}(m,0)\Big|, \thickspace \Big|\hat{F}_{I,I',M}^{\pm}(0, m)\Big| \leq \frac{\Cr{bspoly00}}{m};
    \]
    \item if $m m'\neq 0$, then
    \[
    \Big|\hat{F}_{I,I',M}^{\pm}(m,m')\Big| \leq \frac{\Cr{bspoly00}}{mm'}.
    \]
\end{enumerate}

\end{lem}

The above two lemmas can be viewed as Fourier expansions of the characteristic functions with respect to the basis $\cos(m\theta)$. When we deal with elliptic curves without complex multiplication, it will be useful to change the basis of trigonometric polynomials in Lemma \ref{bspoly1} to the basis of the $m$-th Chebyshev polynomials of the second type $U_m(\cos(\theta))$, which form an orthonormal basis for $L^2([0, \pi], \mu_{ST})$ with respect to the usual inner product $\langle f, g \rangle = \int_0^{\pi} f(\theta)g(\theta)d\mu_{ST}$.
For a demonstration of said base change, see \cite[Section 3]{rouse2017explicit}.

\subsection{Newforms}

Here we briefly recall the notion of newforms. For a complete treatment, see Section 2.5 in \cite{ono2004web}.  For a positive integer $N$, recall that the level $N$ congruence subgroup $\Gamma_0(N)$ is defined by 
$$
\Gamma_0(N):=\left\{ 
\begin{pmatrix}
a & b\\
c & d
\end{pmatrix}
\in\textup{SL}_2(\Z): c\equiv 0 \bmod N\right\}.
$$
Throughout the rest of the paper, let $\mathcal{M}_{k}(\Gamma_0(N), \chi)$ denote the space of modular forms of weight $k$, level $N$, and nebentypus $\chi$, and let $S_{k}(\Gamma_0(N), \chi) \subset \mathcal{M}_{k}(\Gamma_0(N), \chi)$ denote the subspace of cusp forms, both with respect to $\Gamma_0(N)$. When $\chi$ is trivial, write $\mathcal{M}_{k}(\Gamma_0(N), \chi) = \mathcal{M}_{k}(\Gamma_0(N))$, $\mathcal{S}_{k}(\Gamma_0(N), \chi) = \mathcal{S}_{k}(\Gamma_0(N))$. Now, given a positive integer $d$, define the $V$-operator $V(d)$ as
$$
\Bigg(\sum_{n\ge n_0}c(n)q^n\Bigg)|V(d):=\sum_{n\ge n_0}c(n)q^{dn}.
$$
If $f(z)\in S_k(\Gamma_0(N))$ and $d>1$, then both $f(z)$ and $f(dz) = f(z) | V(d)$ are in $S_k(\Gamma_0(dN))$ (Proposition 2.22, \cite{ono2004web}). Thus there are at least two natural ways for a function in $S_k(\Gamma_0(dN))$ to come from lower levels. Motivated by this, we define the \textit{space of oldforms} $S_k^\textup{old}(\Gamma_0(N)) \subset S_k(\Gamma_0(N))$ as
$$
S_k^\textup{old}(\Gamma_0(N)):=\bigoplus_{\substack{dM|N \\ M \neq N}}S_k(\Gamma_0(M))|V(d),
$$
where we sum over pairs of positive integers $(d,M)$ satisfying $dM|N$ and $M\neq N$. Now, let $\mathfrak{F}_N$ denote the fundamental domain for the action of $\Gamma_0(N)$ on the upper half of the complex plane, and recall the \textit{Petersson inner product} between cusp forms $f(z), g(z) \in S_k(\Gamma_0(N))$, defined as
\begin{align*}
    \langle f, g \rangle := \frac{1}{[\textup{SL}_2(\mathbb{Z}) : \Gamma_0(N)]}\int_{\mathfrak{F}_N} f(z)\overline{g(z)}y^{k-2}dxdy
\end{align*}
where $z = x + iy$.

\begin{definition}
    Define $S_k^{\textup{new}}(\Gamma_0(N))$, the \emph{space of newforms}, to be the orthogonal complement of $S_k^\textup{old}(\Gamma_0(N))$ in $S_k(\Gamma_0(N))$ with respect to the Petersson inner product.
\end{definition}

\begin{definition}
    A \emph{newform} in $S_k^{\textup{new}}(\Gamma_0(N))$ is a normalized cusp form that is an eigenform for all the Hecke operators on $S_k^{\textup{new}}(\Gamma_0(N))$, the Atkin-Lehner involution $|_kW(N)$, and all of the Atkin-Lehner involutions $|_kW(Q_p)$ for each prime $p|N$ (for more details, see \cite[Section 2.5]{ono2004web}).
\end{definition}

\subsection{CM elliptic curves, Gr\"{o}ssencharacters, and automorphic induction}\label{secinfo}

Let $K$ be an imaginary quadratic field with absolute discriminant $D_K$, ring of integers $\mathcal{O}_K$, and absolute norm $\mathrm{N}=\mathrm{N}_{K/\mathbb{Q}}$ defined as $\mathrm{N}\abold = |\mathcal{O}_K/\abold|$ for all nonzero ideals $\abold \subset \mathcal{O}_K$.  Let $\chi_{K}$ be the Kronecker character associated to $K$; note that for all primes $p \nmid D_K$,
\begin{align*}
    \chi_K(p) &= \begin{cases} 1 & p \text{ splits in }K \\ -1 & p \text{ inert in }K \end{cases}
\end{align*}
\begin{definition}[\cite{iwaniec1997topics}, Section 12.2]
    Given a $u_\xi \in\mathbb{Z}$, define $\xi_{\infty} : K^\times \rightarrow S^1$ to be the group homomorphism satisfying
    $$
    \xi_\infty(a)=\bigg(\frac{a}{|a|}\bigg)^{u_\xi}.
    $$
    Let $\mathfrak{m}$ be an ideal of $\mathcal{O}_K$. Define a \emph{Hecke Gr\"{o}ssencharacter} $\xi$ with modulus $\mathfrak{m}$ to be a group homomorphism from
    \[
    I_\mbold = \{ \abold \text{ fractional ideal in }\mathcal{O}_K : (\abold, \mbold) = 1\}
    \]
    to 
    $$
    \{z \in \mathbb{C} : |z| = 1\}
    $$ that agrees with $\xi_\infty$ on
    \[
    P_\mbold = \{a\mathcal{O}_K : a \in K^{\times}, a \equiv 1 \bmod \mbold\}.
    \]
\end{definition}
Note that there may exist another Gr\"{o}ssencharacter $\xi^\ast$ of $K$ with modulus $\mbold^\ast \supset \mbold$ which satisfies that $\xi^\ast(\abold) = \xi(\abold)$ for all $\abold \in I_\mbold$. The largest ideal $\nbold$ which is a modulus of such a Gr\"{o}ssencharacter is defined as the conductor of $\xi$. If $\nbold = \mbold$, $\xi$ is called primitive. The following theorem characterizes the $L$-function $L(s,E)$ of an elliptic curve $E/\mathbb{Q}$ with CM over an imaginary quadratic field $K$ as the $L$-function of a Gr{\"o}ssencharacter defined over $K$. 
\begin{thm}[{\cite[Theorem II.10.5]{silverman1994advanced}}]\label{cmisgrossen}

Fix the notation above. Let $E/\mathbb{Q}$ be an elliptic curve with complex multiplication over an imaginary quadratic field $K$ with absolute discriminant $D_K$. Then there exists a primitive Gr\"{o}ssencharacter $\xi$ with conductor $\mathfrak{m} \subset \mathcal{O}_K$ and $u_{\xi}=1$ such that $\textup{N}\mathfrak{m}=N_E/D_K$ and $L(s,E)=L(s,\xi).$

\end{thm}

It will often be convenient to interpret the $L$-functions of Hecke Gr\"{o}ssencharacters as $L$-functions associated with modular forms via automorphic induction.

\begin{thm}[{\cite[Theorem 12.5]{iwaniec1997topics}}]\label{autoinduct}

Let $K$ be an imaginary quadratic field with absolute discriminant $D_K$ and let $\xi$ be a primitive Gr\"{o}ssencharacter with conductor $\mathfrak{m}$ and $u_{\xi}$ non-negative. Let $\chi$ be the Dirichlet character given by $\chi(n) = \chi_D(n)\xi((n))$ for all $n \in \mathbb{Z}^+$. Then the modular form
$$
f(z)=\sum_{\mathfrak{a}}\xi(\mathfrak{a})(\textup{N}\mathfrak{a})^{u_{\xi}/2}e(z\textup{N}\mathfrak{a}) \in \mathcal{M}_{u+1}(\Gamma_0(D_K \textup{N}\mathfrak{m}), \chi)
$$
is a newform and satisfies that
\begin{align*}
    L(s, f) &= L(s, \xi).
\end{align*}When $u_{\xi}>0$, $f$ is a cusp form.   

\end{thm}

Now, given a CM elliptic curve $E$, consider $\xi$ as defined in Theorem \ref{cmisgrossen}. Define $\xi_m (\bmod\text{~}\mathfrak{m}_m)$ to be the primitive Gr\"{o}ssencharacter which induces $\xi^m$, and let $\chi_m$ be the Dirichlet character given by $\chi_m(n)=\chi_D(n)\xi_m((n))$ for all $n \in \mathbb{Z}^+$. Note that $D_K\textup{N}\mbold_m | N_E$. Let $f_m$ be the holomorphic cusp form associated with $\xi_m$ as in Theorem \ref{autoinduct}. Then $f_m \in S_{m + 1}(\Gamma_0(|D_K|\textup{N}\mbold_m), \chi_m)$.

\subsection{Automorphic \emph{L}-functions}

Let $\mathbb{A}_\mathbb{Q}$ be the adele ring of $\mathbb{Q}$. Let $\mathfrak{F}_m$ denote the set of all cuspidal representations of $\text{GL}_m(\mathbb{A}_\mathbb{Q})$ with unitary central character that are trivial on the diagonally embedded copy of the positive real numbers. For $\pi\in\mathfrak{F}_m$, let $\tilde{\pi}$ denote the representation contragredient to $\pi$ and let $q_\pi$ denote the conductor of $\pi$. Now, there exists a standard $L$-function $L(s, \pi)$ attached to $\pi$; let the conductor of this $L$-function be $q_\pi$. The local parameters of this $L$-function are known as the Satake parameters, and for each prime $p$ the Satake parameters $\alpha_{1,\pi}(p), \dots, \alpha_{m,\pi}(p)\in \mathbb{C}$ satisfy that when $p\nmid q_\pi$, $\alpha_{j,\pi}(p)\neq 0$ for all $j$. For every $n \in \mathbb{Z}^+$, denote the $n$th Dirichlet coefficient of $L(s, \pi)$ as $a_\pi(n)$. The Dirichlet series of $L(s, \pi)$ is given by the formula
$$
L(s,\pi)=\prod_{p}\prod_{j=1}^{m}\Big(1-\frac{\alpha_{j,\pi}(p)}{p^s}\Big)^{-1}=\sum_{n=1}^{\infty}\frac{a_\pi(n)}{n^s},
$$
which converges absolutely when $\re(s)> 1$. Define 
$$
\Gamma_\mathbb{R}(s)=\pi^{-s/2}\Gamma(s/2),
$$
where $\Gamma$ is the usual gamma function. The local parameters at infinity, $\mu_\pi(1), \dots, \mu_\pi(n)$, of $L(s, \pi)$, are known as the Langlands parameters. The gamma factor $\gamma(s, \pi)$ of $L(s, \pi)$ is by definition
$$
\gamma(s,\pi)=\prod_{j=1}^{m}\Gamma_\mathbb{R}(s+\mu_\pi(j)).
$$ Note that all of the relevant automorphic representations within this paper will satisfy the Generalized Ramanujan conjecture, so we have the following bounds on the Satake and Langlands parameters,
\begin{equation}
\label{GRC}
|\alpha_{j,\pi}(p)|\le 1,\qquad \re(\mu_\pi(j)) \ge 0.
\end{equation}
Note that as particular examples of automorphic $L$-functions, any Dirichlet $L$-function $L(s,\chi)$ corresponds to a one-dimensional cuspidal representation $\chi$ of $\text{GL}_1(\mathbb{A}_\mathbb{Q})$ and thus is an automorphic $L$-function. 

Recall that $L(s,\pi)$ is always meromorphic; if $\pi$ is the trivial representation $\mathbbm{1}$ of $\text{GL}_1(\mathbb{A}_\mathbb{Q})$ then $L(s, \pi)$ has a simple pole at $s = 1$, otherwise $L(s, \pi)$ is entire. Denote the order of the pole of $L(s,\pi)$ at $s=1$ as $r_\pi$. Then $r_\pi=1$ when $\pi=\mathbbm{1}$ and $r_\pi = 0$ otherwise. Now, the complete $L$-function 
$$
\Lambda(s,\pi)=(s(s-1))^{r_\pi}q_\pi^{s/2}L(s,\pi)\gamma(s, \pi)
$$
is entire of order 1. Moreover, there exists a $W(\pi) \in \mathbb{C}$ satisfying $|W(\pi)| = 1$, for which the functional equation
$$
\Lambda(s,\pi)=W(\pi)\Lambda(1-s,\tilde{\pi}),
$$
holds true. Note that $q_{\tilde{\pi}} = q_{\pi}$, and the following sets are equal:
\[
    \{\alpha_{j, \tilde{\pi}}(p)\} = \{\overline{\alpha_{j, \pi}(p)}\}, \qquad \{\mu_{\tilde{\pi}}(j)\} = \{\overline{\mu_{\pi}(j)}\}.
\]
The analytic conductor of $L(s,\pi)$ is defined as 
$$
C(\pi)=q_\pi\prod_{j=1}^{m}(3+|\mu_\pi(j)|).
$$
Note that throughout this paper, we will use $C(\cdot)$ to generally refer to the analytic conductor of an $L$-function (see \cite[Page 95]{iwaniec2021analytic} for the definition).
Moreover, we will define $\Lambda_\pi(n)$ to be the $n$th coefficient of the logarithmic derivative of $-L(s, \pi)$. Specifically, we have the formula
$$
\sum_{n=1}^{\infty}\frac{\Lambda_\pi(n)}{n^s}=-\frac{L'}{L}(s,\pi)=\sum_p\sum_{\ell=1}^{\infty}\frac{\sum_{j=1}^{m}\alpha_{j,\pi}(p)^\ell\log p}{p^{\ell s}}.
$$

\subsection{Rankin-Selberg \emph{L}-functions}

Given any two $\pi\in\mathfrak{F}_m$ and $\pi'\in\mathfrak{F}_{m'}$, it is often useful to consider the Rankin-Selberg convolution of their $L$-functions. This convolution is an $L$-function itself, with Satake parameters denoted as $\alpha_{j,j',\pi\times \pi'}(p)$. A complete description of these parameters is given in \cite[Appendix]{soundararajan2019weak}. Note that if $m' = m$ and $\pi' = \tilde{\pi}$, then we call the resulting Rankin-Selberg convolution the Rankin-Selberg square of $\pi$. Note that for any prime $p\nmid q_\pi q_{\pi'}$, we have that
$$
\{\alpha_{j,j',\pi\times \pi'}(p)\}=\{\alpha_{j,\pi}(p)\alpha_{j',\pi'}(p)\}.
$$
The Dirichlet series of the Rankin-Selberg convolution of $L(s,\pi)$ and $L(s,\pi')$ is given by
$$
L(s,\pi\times \pi')=\prod_p\prod_{j=1}^{m}\prod_{j'=1}^{m'}\Big(1-\frac{\alpha_{j,j',\pi\times \pi'}(p)}{p^{s}}\Big)^{-1}=\sum_{n=1}^
{\infty}\frac{a_{\pi\times\pi'}(n)}{n^s},
$$
is associated to the tensor product $\pi \times \pi'$, and converges absolutely for $\re(s)>1$. Let $q_{\pi\times\pi'}$ be the conductor of $L(s,\pi\times\pi')$. Note that $q_{\pi\times\pi'}|q_{\pi}^{m'}q_{\pi'}^m$ (see \citep{bushnell1997upper}). Now, denote the Langlands parameters of $L(s, \pi \times \pi')$ as $\mu_{\pi\times\pi'}(j,j')\in \mathbb{C}$; a complete description of these Langlands parameters can be found in \cite[Proof of Lemma 2.1]{soundararajan2019weak}. By definition, the gamma factor of $L(s, \pi \times \pi)$ is given by 
$$
\gamma(s,\pi\times\pi')=\prod_{j=1}^{m}\prod_{j'=1}^{m'}\Gamma_\mathbb{R}(s+\mu_{\pi\times\pi'}(j,j')).
$$
It is known that $L(s,\pi\times\pi')$ is entire if $\pi'\neq \tilde{\pi}$. If $\pi' = \tilde{\pi}$, $L(s,\pi\times\pi')$ has a simple pole at $s=1$ and is holomorphic elsewhere. Let $r_{\pi\times\pi'}$ denote the order of the pole of $L(s,\pi\times\pi')$ at $s=1$; then $r_{\pi\times\pi'}=1$ if $\pi$ and $\pi'$ are dual to each other and $r_{\pi\times\pi'}=0$ otherwise. The completed $L$-function of $L(s, \pi \times \pi')$ is given by
$$
\Lambda(s,\pi\times\pi')=(s(s-1))^{r_{\pi\times\pi'}}q_{\pi\times\pi'}^{s/2}L(s,\pi\times\pi')\gamma(s,\pi\times\pi').
$$
Note that $\Lambda(s,\pi\times\pi')$ is entire and of order $1$. Moreover, there exists a $W(\pi\times\pi')\in\mathbb{C}$ satisfying $|W(\pi\times\pi')| = 1$ such that the functional equation
$$
\Lambda(s,\pi\times\pi')=W(\pi\times\pi')\Lambda(1-s,\tilde{\pi}\times\tilde{\pi'}),
$$
holds. Define the analytic conductor  of $L(s, \pi \times \pi')$ as
$$
C(\pi\times\pi')=q_{\pi\times\pi'}\prod_{j=1}^{m}\prod_{j'=1}^{m'}(3+|\mu_{\pi\times\pi'}(j,j')|).
$$
By the work of Bushnell and Henniart \citep{bushnell1997upper} and Brumley \cite[Appendix]{humphries2019standard}, we have the following bound for the analytic conductor
\begin{align}\label{conductorbound}
\log C(\pi\times\pi') \ll m' \log C(\pi)+m\log C(\pi')
\end{align}
Finally, we denote the $n$th coefficient of the negative log derivative of $L(s, \pi \times \pi')$ as $\Lambda_{\pi \times \pi'}(n)$. By definition, we have
$$
\sum_{n=1}^{\infty}\frac{\Lambda_{\pi\times\pi'}(n)}{n^s}=-\frac{L'}{L}(s,\pi\times\pi')=\sum_p\sum_{\ell=1}^{\infty}\frac{\sum_{j=1}^{m}\alpha_{j,\pi\times\pi'}(p)^\ell\log p}{p^{\ell s}}.
$$

\subsection{Isobaric sums} Here we recall some basic facts about the isobaric sum operation $\boxplus$, first introduced by Langlands in \cite{langlands1977isobaric}. Let $k \geq 1$ be an integer, let $m_1, \ldots m_k \geq 1$ be integers, let $\pi_i \in \mathfrak{F}_{m_i}$, let $r = \sum_{i=1}^k m_i$, and let $t_1, t_2, \ldots t_k \in \mathbb{R}$. Consider the isobaric automorphic representation $\Pi$ of $\GL_r(\mathbb{A})$, defined by
\begin{align*}
    \Pi &= \pi_1 \otimes |\det|^{it_1} \boxplus \ldots \boxplus \pi_k \otimes |\det|^{it_k}
\end{align*}
The $L$-function associated to $\Pi$ is
\begin{align*}
    L(s, \Pi) &= \prod_{j=1}^k L(s + it_j, \pi_j).
\end{align*}
The analytic conductor of this $L$-function is defined as
\begin{align*}
    C(\Pi, t) &= \prod_{j=1}^k C(\pi_j, t + t_j) \\
    C(\Pi) &= C(\Pi, 0).
\end{align*}
Now, let $k' \geq 1$ be an integer, let $m'_1, \ldots m'_{k'} \geq 1$ be integers, let $\pi'_i \in \mathfrak{F}_{m'_i}$, let $r' = \sum_{i=1}^{k'} m_i'$, and let $t'_1, \ldots t'_{k'} \in \mathbb{R}$. Consider the isobaric automorphic representation $\Pi' \in \GL_{r'}(\mathbb{A})$, defined by
\begin{align*}
    \Pi &= \pi'_1 \otimes |\det|^{it'_1} \boxplus \ldots \boxplus \pi_{k'} \otimes |\det|^{it'_{k'}}.
\end{align*}
Then, the Rankin-Selberg convolution of $L(s, \pi)$ and $L(s, \pi')$ is given by
\begin{align*}
    L(s, \Pi \times \Pi') &= \prod_{j=1}^k \prod_{j'=1}^{k'} L(s + it_j + it'_{j'}, \pi_j \times \pi'_{j'})
\end{align*}
and has analytic conductor
\begin{align*}
    C(\Pi \times \Pi', t) &= \prod_{j=1}^d \prod_{j'=1}^{d'} C(\pi_j \times \pi'_{j'}, t + t_j + t'_{j'}) \\
    C(\Pi \times \Pi') &= C(\Pi \times \Pi', 0).
\end{align*}
In our establishment of zero-free regions below we will implicitly need the following lemma, which is evident from \cite[Section A]{soundararajan2019weak}:
\begin{lem}
    Given any unitary isobaric automorphic representation $\Pi$ with $L$-function $L(s, \pi)$, the Dirichlet coefficients of $-\frac{L'}{L}(s, \Pi \times \tilde{\Pi})$ are all nonnegative.
\end{lem}

\subsection{Symmetric power \emph{L}-functions}\label{ssympowerinfo}

A particular type of $L$-function of interest to us is the symmetric power $L$-function. Consider an elliptic curve $E/\mathbb{Q}$. Recall that by the modularity theorem \citep{breuil2001modularity}, there exists a non-CM cusp form $f(z)\in S_2^\text{new}(\Gamma_0(N_E))$ of weight $2$ and level $N_E$ corresponding to $E$. Now, let the Fourier expansion of $f$ be $f(z)=\sum_{n=1}^{\infty}a_f(n)e^{2\pi inz}$. We have that $a_f(p)$ agrees with the trace of Frobenius of $E$ modulo $p$, so thus we may write $a_f(p) = 2\sqrt{p}\cos\theta_E(p)$ (where $\theta_E(p)$ is as defined in Section \ref{sintro}).

For each non-negative integer $m$ and prime $p$ consider the Satake parameters $\alpha_{0,\Sym^m E}(p),\dots, \\ \alpha_{m,\Sym^mf}(p)\in \mathbb{C}$. When $p \nmid N_E$, the Satake parameters satisfy the useful identity
$$
\{\alpha_{0,\Sym^m E}(p),\ldots,\alpha_{m,\Sym^m E}(p)\}=\{e^{i(m-2j)\theta_E(p)}\colon 0\leq j\leq m\}.
$$
A complete description of the values of $\alpha_{j,\Sym^m E}(p)$ can be derived from \cite[Appendix]{soundararajan2019weak}. Now, the $m$th symmetric power $L$-function associated to $f$ (denoted $L(s, \Sym^m f) = L(s, \Sym^m E)$) is the $L$-function with local parameters
\[
\{\alpha_{0,\Sym^m f}(p),\dots, \alpha_{m,\Sym^mf}(p)\}
\]
If we denote the $n$th Dirichlet coefficient of the $m$th symmetric power $L$-function as $a_{\Sym^m E}(n)$, then by definition we have the Euler product and Dirichlet series expansions
$$
L(s,\Sym^m E)=\prod_p\prod_{j=0}^{m}\Big( 1-\frac{\alpha_{j,\Sym^m E}(p)}{p^s}\Big)^{-1}=\sum_{n=1}^{\infty}\frac{a_{\Sym^m E}(n)}{n^s},
$$
which converge for $\re(s)>1$. When $p \nmid N_E$, we may readily compute $a_{\Sym^m E}(p)=U_m(\cos\theta_p)$, where $U_m$ is the $m$-th Chebyshev polynomial of the second kind. Note that $L(s, \Sym^m E)$ is a self-dual $L$-function.

Define $\Gamma_\mathbb{C}(s):=\Gamma_{\mathbb{R}}(s)\Gamma_{\mathbb{R}}(s+1)=2(2\pi)^{-s}\Gamma(s)$, let $q_{\Sym^m E}$ be the conductor of $L(s, \Sym^m E)$, and for even $m$, let $r \in \{0, 1\}$ such that $r \equiv m/2 \bmod 2$.
The gamma factor of $L(s, \Sym^m E)$ is given by 
$$
\gamma(s,\Sym^m E)=\begin{cases} \prod_{j=1}^{(m+1)/2}\Gamma_\mathbb{C}(s+(j-\frac{1}{2})) & \text{if } m\text{ is odd } \\
\Gamma_\mathbb{R}(s+r)\prod_{j=1}^{m/2}\Gamma_\mathbb{C}(s+j) & \text{if } m\text{ is even. }
\end{cases}
$$
One may define the analytic conductor $C(\Sym^m E)$ of $L(s,\Sym^m E)$ similarly to the previous subsections.
Note that the complete $L$-function of $L(s, \Sym^m E)$ is entire of order $1$, and is given by
$$
\Lambda(s,\Sym^m E)=q_{\Sym^m E}^{s/2}\gamma(s,\Sym^m E)L(s,\Sym^m E).
$$
Also, there exists a $W(\Sym^m E) \in \mathbb{C}$ satisfying $|W(\Sym^m E)| = 1$ for which the functional equation
$$
\Lambda(s,\Sym^m E)=W(\Sym^m E)\Lambda(1-s,\Sym^m E)
$$
holds.

Let $\pi_f\in\mathfrak{F}_2$ be the cuspidal form corresponding to $f$. Then as detailed in \cite[Theorem 6.1]{thorner2021effective}, due to the work by Newton and Thorne (\cite[Theorem B]{thorne2021symmetric} and \cite[Theorem A]{newton2021symmetric}) and the work in \citep{moreno1985functions} and \citep{cogdell2004complex}, we know that $L(s,\Sym^m E)$ is the standard $L$-function associated to the representation $\Sym^m \pi_f\in \mathfrak{F}_{m+1}$ with the same gamma factor, complete $L$-function, and functional equation (thus henceforth we will often write $\Sym^m E = \Sym^m \pi_f$). Now, by \cite[Section A.2]{chantaldavid}, we have 
$$
\log q_{\Sym^m E}\ll m\log N_E.
$$
Through a straightforward calculation using Stirling's formula and the above properties, we may now obtain
\begin{align}\label{conductorbound2}
    \log C(\Sym^m E)\ll m\log(N_Em).
\end{align}

\section{Zero-free regions}\label{szerofree}

Throughout this section we will fix the notation from Sections \ref{sintro} and \ref{sinfo}. The following proposition, taken from \cite{thorner2021effective}, will be useful in establishing many of the zero-free regions used in this paper.
\begin{prop}[{\cite[Proposition 4.1]{thorner2021effective}}]\label{prop41}
    Let $\Pi$ be an isobaric automorphic representation of $\GL_r(\mathbb{A})$. If $L(s, \Pi \times \tilde{\Pi})$ has a pole of order $r_{\Pi \times \tilde{\Pi}} \geq 1$ at $s = 1$, then $L(1, \Pi \times \tilde{\Pi}) \neq 0$, and there exists a constant $\Cl[abcon]{prop41}$ such that $L(s, \Pi \times \tilde{\Pi})$ has at most $r_{\Pi \times \tilde{\Pi}}$ real zeroes in the interval
    \begin{align*}
        s &\geq 1 - \frac{\Cr{prop41}}{(r_{\Pi \times \tilde{\Pi}} + 1)\log C(\Pi \times \tilde{\Pi})}.
    \end{align*}
\end{prop}

We will also use the following zero-free region throughout the proof of Theorem \ref{master}.

\begin{thm}[{\cite[Corollary 4.2]{thorner2021effective}}]\label{cor42}
    Let $\pi \in \mathfrak{F}_m$ and $\pi' \in \mathfrak{F}_{m'}$. Suppose that both $\pi$ and $\pi'$ are self-dual. There exists a constant $\Cl[abcon]{thorner} > 0$ for which the following results hold.
    \begin{enumerate}
        \item $L(s, \pi) \neq 0$ in the region
        \begin{align*}
            \re(s) &\geq 1 - \frac{\Cr{thorner}}{m\log(C(\pi)(3 + |\im(s)|)}
        \end{align*}
        apart from at most one zero. If the exceptional zero exists, then it is real and simple.
        \item $L(s, \pi \times \pi') \neq 0$ in the region
        \begin{align*}
            \re(s) &\geq 1 - \frac{\Cr{thorner}}{(m + m')\log(C(\pi)C(\pi')(3 + |\im(s)|)^{\min(m, m')})}
        \end{align*}
        apart from at most one zero. If the exceptional zero exists, then it is real and simple.
    \end{enumerate}
\end{thm}

To remove the possibility of an exceptional zero, we will frequently use the following proposition.

\begin{prop}\label{siegelzeroprop}
Let $f$ be a holomorphic newform with complex multiplication, and let $\chi$ be the Dirichlet character such that $f=f\otimes\chi$. Let $\pi_f \in \mathfrak{F}_2$ be the representation corresponding with $f$. Let $\pi$ correspond with a cuspidal automorphic representation of $\mathfrak{F}_m$, and suppose that $\pi\otimes\chi\neq \pi$.  Then there exists an absolute constant $\Cl[abcon]{propB}$ such that $L(s,\pi_f \times \pi)$ has no zeroes in the interval 
\begin{align*}
    \re(s) &\geq 1 - \frac{\Cr{propB}}{m\log (C(\pi_f)C(\pi)(|\Im(s)|+3))}.
\end{align*}
\end{prop}
\begin{proof}

Consider the isobaric sum
\[
\Pi = \pi\boxplus \pi\otimes\chi\boxplus \widetilde{\pi}_f\otimes|\cdot|^{-i\gamma}.
\]
By hypothesis, $\pi_f\otimes\chi=\pi_f$ and $\pi\otimes\chi\neq\pi$.  This implies that $\widetilde{\pi}_f\otimes\bar{\chi}=\widetilde{\pi}_f$ and $\pi\otimes\bar{\chi}\neq\pi$.  It follows that $L(s,\Pi\times\widetilde{\Pi})$ factors as
\begin{align*}
&L(s,\pi_f\times\widetilde{\pi}_f) L(s,\pi\times\tilde{\pi})^2 L(s,\pi\times(\widetilde{\pi}\otimes\chi)) L(s,\pi\times(\widetilde{\pi}\otimes\overline{\chi})) L(s+i\gamma,\pi\times\pi_f)L(s-i\gamma,\widetilde{\pi}\times\widetilde{\pi}_f)\\
\times &L(s+i\gamma,\pi\times(\pi_f\otimes\chi))L(s-i\gamma,\widetilde{\pi}\times(\widetilde{\pi}_f\otimes\overline{\chi}))\\
=&L(s,\pi_f\times\widetilde{\pi}_f) L(s,\pi\times\tilde{\pi})^2 L(s,\pi\times(\widetilde{\pi}\otimes\chi)) L(s,\pi\times(\widetilde{\pi}\otimes\overline{\chi})) L(s+i\gamma,\pi\times\pi_f)^2 L(s-i\gamma,\widetilde{\pi}\times\widetilde{\pi}_f)^2.
\end{align*}
This $L$-function has a pole of order $3$ at $s = 1$ (due to the $L(s, \pi_f \times \tilde{\pi}_f)$ and $L(s, \pi \times \tilde{\pi})$ terms). Moreover, if $\beta$ is a real zero of $L(s+i\gamma,\pi\times\pi_f)$, then $\beta$ is also a real zero of $L(s-i\gamma,\widetilde{\pi}\times\widetilde{\pi}_f)$. Hence a zero $\beta + i\gamma$ of $L(s, \pi \times \pi_f)$ would imply a real zero of order $4$ of $L(s, \Pi \times \widetilde{\Pi})$. If $\beta + i\gamma$ is in the interval specified above, then this would now contradict Proposition \ref{prop41} by a simple conductor calculation using (\ref{conductorbound}). Note that the contribution from $C(\chi)$ may be neglected as the conductor of $\chi$ divides the square of the conductor of $\pi_f$, by \cite[Theorem A]{ramakrishnan2021conductor}. 
\end{proof}

Note that in both applications of this Proposition within this paper, the newform $f$ is taken to correspond to an elliptic curve with CM. Hence there are finitely many choices for $\chi$ (as it must be the Kronecker character of the field over which the elliptic curve has CM), and so the contribution of $C(\chi)$ above can be absorbed into the constant $\Cr{propB}$.

\subsection{Non-CM newforms twisted by Dirichlet characters}\label{snonCMzerofree}

\begin{lem}\label{zerofree}

Let $\pi\in\mathfrak{F}_{m+1}$ ($m\ge 1$), and $\chi$ be a primitive Dirichlet character. Suppose that $\pi$ is self dual. Then there exists an absolute constant $\Cl[abcon]{41}$ such that $L(s,\pi\otimes\chi)\neq 0$ for all $s$ in the region 
\begin{align*}
    \textnormal{Re}(s)\ge 1-\frac{\Cr{41}}{m\log(C(\pi)C(\chi)(3+|\textnormal{Im}(s)|))},
\end{align*}
apart from at most one Siegel zero. If such a Siegel zero exists, it is real and simple, and $\chi$ is self-dual.

\end{lem}

\begin{proof}

When $\chi$ is real, this is a special case of Theorem \ref{cor42}. Hence assume $\chi$ is not real. Suppose for the sake of contradiction that $\rho=\sigma+it$ is a zero in that region. Let $\Pi=\chi|\cdot|^{it}\boxplus \overline{\chi} |\cdot|^{-it}\boxplus \pi$. If $\psi$ denotes the primitive Dirichlet character that induces $\chi^2$, then
$$
L(s,\Pi\times \overline{\Pi})=\zeta(s)^2 L(s,\pi\times\overline{\pi})L(s+it, \pi\otimes\chi)^2L(s-it, \pi\otimes\overline{\chi})^2 L(s+2it, \psi)L(s-2it, \overline{\psi}),
$$
which has exactly 3 poles at $s=1$. However, since $\pi$ is self-dual, both $L(s+it, \pi\otimes\chi)$ and $L(s-it, \pi\otimes\overline\chi)$ have a zero at $\sigma$, so thus $L(s,\Pi\times \tilde\Pi)$ has at least four zeroes at $\sigma$. Combined with the bound on the conductor in (\ref{conductorbound}), this now contradicts Proposition \ref{prop41}. Note that this proof works even when $t = 0$; hence a Siegel zero cannot exist in this case.
\end{proof}

\begin{lem}\label{apnozero}

Consider a non-CM elliptic curve $E$, and let $f$ be the cusp form corresponding to $E$. Let $\pi_f$ be the representation corresponding to $f$, and let $\pi = \textup{Sym}^m \pi_f \in \mathfrak{F}_{m+1}$. Then we may replace $\Cr{41}$ in Lemma \ref{zerofree} with an absolute constant $\Cl[abcon]{42}$ such that Lemma \ref{zerofree} holds true for $\pi$ without the possibility of a Siegel zero.

\end{lem}

\begin{proof}

We only need to consider the case when $\chi$ is self-dual. Consider the isobaric representation $\Pi_m=\one \boxplus \text{Sym}^2 \pi_f \boxplus (\text{Sym}^m \pi_f\otimes\chi)$. Recall the identities
\begin{align*}
    L(s,\Sym^m \pi_f \times \Sym^m \pi_f)&= \zeta(s) \prod_{j=1}^m L(s,\Sym^{2j}\pi_f),\\
    \Sym^m \pi_f \times \Sym^2 \pi_f &= \boxplus_{j=0}^2 \Sym^{m + 2 - 2j} \pi_f,\\
\end{align*}
These identities imply that
\begin{align*}
    L(s, \Pi_m\times\overline\Pi_m)=\zeta(s)^3L(s,\text{Sym}^m \pi_f \otimes\chi)^4&L(s,\text{Sym}^2\pi_f)^3L(s,\text{Sym}^4 \pi_f)L(s,\text{Sym}^{m+2} \pi_f\otimes\chi)^2 \\&\times L(s,\text{Sym}^{m-2} \pi_f\otimes\chi)^2\prod_{j=1}^{m}L(s,\text{Sym}^{2j}\pi_f).
\end{align*}
Now, this function has a pole of order 3 at $s=1$. If  $L(s,\text{Sym}^m \pi_f\otimes\chi)$ has a Siegel zero at $\rho$, $L(s, \Pi_m\times\overline\Pi_m)$ will have a zero of order at least 4 at $\rho$. Combined with the bound on the conductor in (\ref{conductorbound}), this would again contradict Proposition \ref{prop41}. \end{proof}

\subsection{CM newforms twisted by Dirichlet characters}\label{sCMzerofree}

Consider an elliptic curve $E$ which has complex multiplication over an imaginary quadratic field $K$. Let $m \geq 1$ be an integer, $\xi$ denote the primitive Gr\"{o}ssencharacter which satisfies $L(s, E) = L(s, \xi)$, $f_m$ denote the cusp form which induces $\xi_m$ (where $\xi_m$ is as defined in Section \ref{secinfo}), and $\pi_{f_m} \in \mathfrak{F}_2$ denote the representation corresponding to $f_m$.  Finally, let $\chi$ be a primitive Dirichlet character that induces a character $\bmod \thinspace q$.
\begin{lem}\label{CMzerofree}
     There exists an absolute constant $\Cl[abcon]{cmapzerofree} > 0$ such that $L(s, \pi_{f_m} \otimes \chi)$ has no zeroes in the region
    \begin{align*}
        \re(s) &\geq 1 - \frac{\Cr{cmapzerofree}}{\log(N_Eqm(|3+\im(s)|))}.
    \end{align*}
\end{lem}
\begin{proof}
    By \cite[Theorem 7.5]{iwaniec1997topics}, there exists an integer $M \leq N_Eq^2$ such that $f_m \otimes \chi$ is a cusp newform in $S_{m+1}(\Gamma_0(M), \chi_m\chi^2)$ (where $\chi_m$ is as defined in Section \ref{secinfo}). Hence we have that $C(\pi_{f_m} \otimes \chi) \ll N_Eq^2m^2$. Moreover, by \cite[Section 5.11]{iwaniec2021analytic}, we have that the Rankin-Selberg convolutions $L(s, \pi_{f_m} \otimes \chi \times \pi_{f_m} \otimes \chi)$ and $L(\pi_{f_m} \otimes \chi \times \overline{\pi_{f_m} \otimes \chi})$ both exist, with the former being entire and the latter having a simple pole at $s = 1$. Now, using the information above and from Section \ref{secinfo}, we may apply \cite[Theorem 5.10]{iwaniec2021analytic}, which proves the above except for the possibility of an exceptional zero. We may now remove this possibility using \cite[Theorem C]{hoffstein1995siegel}.
\end{proof}

\subsection{Rankin-Selberg convolutions of a CM and Non-CM newform}\label{sCMnonCMzerofree}

Consider two elliptic curves $E, E'$, the former having complex multiplication over an imaginary quadratic field $K$, and the latter without complex multiplication over any imaginary quadratic field. Let $\xi$ denote the primitive Gr\"{o}ssencharacter which satisfies $L(s, E) = L(s, \xi)$, $g_m$ (instead of $f_m$) denote the cusp form which induces $\xi_m$, and $\pi_{g_m} \in \mathfrak{F}_2$ denote the representation corresponding to $g_m$. Also let $f'$ denote the cusp form which corresponds to $E'$, and let $\pi_{f'} \in \mathfrak{F}_2$ be the representation which corresponds to $f'$.

\begin{lem}\label{zerofreegrossen}
There exists an absolute constant $\Cl[abcon]{461} > 0$ such that for all integers $1\le m, m' \le M$, \\ $L(s, \pi_{g_{m}} \times \Sym^{m'} \pi_{f'}) \neq 0$ in the region 
\begin{align*}
    \textup{Re}(s) &\geq 1- \frac{\Cr{461}}{M^2\log(N_{E}N_{E'}(3+|\textup{Im}(s)|))}.
\end{align*}
\end{lem}
\begin{proof}
    When $\im(s) \neq 0$, we may apply Theorem \ref{cor42}, while in the case where $\im(s) = 0$, we may apply Proposition \ref{siegelzeroprop}. In both cases, the result follows from a conductor calculation using (\ref{conductorbound2}) and the properties discussed in Section \ref{secinfo}.
\end{proof}

\subsection{Rankin-Selberg convolutions of two CM newforms}\label{sCMCMzerofree} 
Consider two twist-inequivalent elliptic curves $E, E'$ having complex multiplication over two imaginary quadratic fields $K, K'$ (respectively). Let $m, m' \geq 1$ be integers. Consider the Gr\"{o}ssencharacters $\xi, \xi'$ corresponding to $E, E'$, let $f = f_m, f' = f'_{m'}$ be the cusp forms corresponding to $\xi_m, \xi'_{m'}$ (as defined in \ref{secinfo}), and let $\pi_f, \pi_{f'} \in \mathfrak{F}_2$ denote the representations corresponding to $f, f'$ respectively. Finally, let $\mathfrak{q} = N_EN_{E'}mm'$.
\begin{lem}\label{CMCMsiegelzero}
    There exists an absolute constant $\Cl[abcon]{CMCMzerofree} > 0$ such that $L(s, \pi_f \times \pi_{f'}) \neq 0$ in the region
\begin{align*}
    \textup{Re}(s) &\geq 1 - \frac{\Cr{CMCMzerofree}}{\log(\mathfrak{q}(3 + |\textup{Im}(s)|))}.
\end{align*}
\end{lem}

\begin{proof}
    Note that as a special case of the work due to Bushnell and Henniart \cite{bushnell1997upper}, we know that $q(\pi_f \times \pi_{f'})$ divides $q(\pi_{f})^2q(\pi_{f'})^2$ (where $q(\cdot)$ denotes the conductor), and hence $C(\pi_{f} \times \pi_{f'}) \ll \mathfrak{q}^4$. Now, since $\pi_f, \pi_{f'}$ are both self-dual, Lemma \ref{CMCMsiegelzero} follows from Theorem \ref{cor42} with the exception of a possible Siegel zero. We may remove this possibility using Proposition \ref{siegelzeroprop}.
\end{proof}

\section{Prime Number Theorems}\label{sPNT}

Throughout this section, fix the notation from Sections \ref{sintro} and \ref{sinfo}. 

\subsection{Representations of \texorpdfstring{$\GL_m(\mathbb{A})$}{GLm(A)} twisted by Dirichlet characters} 

\begin{prop}\label{PNTnonCMAP}

Let $\pi\in\mathfrak{F}_m$ be self-dual, and let $\chi$ be a primitive Dirichlet character. Let $\beta_1$ denote the possible Siegel zero from Lemma \ref{zerofree}. For each prime $p$, let $a_\pi(p)$ be the coefficient of $p^{-s}$ in the Dirichlet series expansion of $L(s, \pi )$. Suppose that $\pi$ satisfies the generalized Ramanujan conjecture at all primes $p\nmid C(\pi)$ and all its Langlands parameters satisfy either $\mu_\pi(j)=0$ or $\textnormal{Re}(\mu_\pi(j))\ge \frac{1}{2}$ for each $j$. Then there exists an absolute constant $\Cl[abcon]{43} > 2$ such that if $2\le (C(\pi)C(\chi))^m\le x^{1/\Cr{43}}$ and $\Cr{43}mx^{-\frac{1}{\Cr{43}m}}<\frac{1}{4}$, we have \begin{align*}
    &\Bigg| \bigg( \sum_{\substack{p\le x\\ p\nmid q_\pi q_\chi}}a_{\pi}(p)\chi(p)\log p \bigg) -r_{\pi \otimes \chi} x+\frac{x^{\beta_1}}{\beta_1} \Bigg| \\&\ll m^2x^{1-\frac{1}{\Cr{43}m}}+m^2x\Big(\exp\Big(-\frac{\Cr{41}\log x}{2m\log C(\pi)C(\chi)}\Big)+\exp\Big(-\frac{\sqrt{\Cr{41}\log x}}{2\sqrt{m}}\Big)\Big).
\end{align*}

where the $\frac{x^{\beta_1}}{\beta_1}$ term is omitted if $\beta_1$ does not exist.

\end{prop}

\begin{proof}

The proof is identical to the proof of \cite[Proposition 5.1]{thorner2021effective}, except for the zero-free region \cite[Corollary 4.2]{thorner2021effective} is replaced by Lemma \ref{zerofree}. Note that by definition, the coefficient $a_{\pi \otimes \chi}(p)$ of $p^{-s}$ in the Dirichlet series expansion of $L(s, \pi)$ is precisely $a_\pi(p)\chi(p)$ for $p \nmid q_{\pi}q_{\chi}$. \end{proof}

\begin{lem}\label{nonCMPNT} Let $E$ be a non-CM elliptic curve, and let $\chi$ be a Dirichlet character which induces a character of modulus $q$. Then, there exists a sufficiently small absolute constant $\Cl[abcon]{2aM}$ such that if $M=\frac{\Cr{2aM}\sqrt{\log x}}{\log(N_Eq\log x))}$, then for all $1 \leq m \leq M,$
    \begin{align*}
    \sum_{\substack{p\le x\\ p\nmid N_E}}\chi(p)U_m(\cos \theta_E(p)) \log p \ll m^2x \Big( x^{-\frac{1}{\Cr{43}m}}+\Big(\exp\Big(-\frac{\Cr{42}\log x}{2m^2\log(N_Eqm)}\Big)+\exp\Big(-\frac{\Cr{42}\sqrt{\log x}}{2\sqrt{m}} \Big)\Big) \Big).
\end{align*}
\end{lem}

\begin{proof} Consider any integer $m \geq 1$. Let $f$ be the cusp form corresponding to $E$, and let $\pi_f \in \mathfrak{F}_2$ be the representation corresponding to $f$. We first establish the following properties about $L(s, \Sym^m \pi_f \otimes \chi)$.
\begin{enumerate}
    \item The conductor of $\text{Sym}^m \pi_f\otimes\chi$ satisfies $\log C(\text{Sym}^m \pi_f\otimes\chi)\ll m\log (N_Eqm)$.
    \item All of the Langlands parameters at infinity of $L(s, \text{Sym}^m \pi_f\otimes \chi)$ are nonnegative, and either integers or half integers.
    \item $L(s, \text{Sym}^m \pi_f\otimes \chi)$ is the $L$-function of a cuspidal automorphic representation in $\mathfrak{F}_{m+1}$.
    \item $L(s, \text{Sym}^m \pi_f\otimes \chi)$ is entire for $m\ge 1$.
    \item $L(s, \text{Sym}^m \pi_f\otimes \chi)$ has no zero in the region 
    $$
    \textnormal{Re}(s)\ge 1-\frac{\Cr{41}}{m^2\log(N_Eqm(3+|\textnormal{Im}(s)|))}.
    $$
    \item $L(s, \text{Sym}^m \pi_f\otimes \chi)$ satisfies the generalized Ramanujan conjecture (GRC). 
\end{enumerate}

(1) follows from (\ref{conductorbound}) and (\ref{conductorbound2}). (2), (3), and (4) all follow from the work due to Newton and Thorne in \citep{newton2021symmetric} and \citep{thorne2021symmetric}. (5) follows from Lemma \ref{zerofree} and Lemma \ref{apnozero}. (6) follows from the definition of the symmetric power $L$-function, given in Section \ref{ssympowerinfo}.

Now, given (1)-(6), we have that evidently there exists a sufficiently small absolute constant $\Cr{2aM}$ such that if $M=\frac{\Cr{2aM}\sqrt{\log x}}{\log(N_Eq\log x))}$, then for all $1 \leq m \leq M$, $\text{Sym}^m \pi_f\otimes \chi\in\mathfrak{F}_{m+1}$ satisfies the conditions in Proposition \ref{PNTnonCMAP} for all $x \geq 3$. Applying the proposition, we thus obtain
\begin{align*}
    \sum_{\substack{p\le x\\ p\nmid N_E}}\chi(p)U_m(\cos \theta_E(p))\log p \ll m^2x^{1-\frac{1}{\Cr{43}m}}+m^2x\Big(\exp\Big(-\frac{\Cr{42}\log x}{2m^2\log(N_Eqm)}\Big)+\exp\Big(-\frac{\Cr{42}\sqrt{\log x}}{2\sqrt{m}}\Big) \Big),
\end{align*} as desired.
\end{proof}

\subsection{CM Newforms twisted by Dirichlet characters}

Consider an elliptic curve $E$ which has complex multiplication over an imaginary quadratic field $K$. Let $m \geq 1$ be an integer, $\xi$ denote the primitive Gr\"{o}ssencharacter which satisfies $L(s, E) = L(s, \xi)$, $f_m$ denote the cusp form which induces $\xi_m$ (where $\xi_m$ is as defined in Section \ref{secinfo}), and $\pi_{f_m} \in \mathfrak{F}_2$ denote the representation corresponding to $f_m$.

\begin{lem}\label{coslogp}
Given any positive integer $m$, there exists an absolute constant $\Cl[abcon]{44} > 0$ such that for all $x \geq N_E$,
\begin{align*}
    \sum_{\substack{p \leq x \\ p \equiv a \bmod q \\ p \textup{ splits in }K}} \cos(m\theta_E(p)) \log p &\ll x\Big(\exp \Big (\frac{-\Cr{44}\log x}{\sqrt{\log x} + 3\log(N_E q m)} \Big ) \log (x N_E q m)^4 \Big ).
\end{align*}
\end{lem}
\begin{proof}
Consider a primitive Dirichlet character $\chi$ which induces a Dirichlet character mod $q$. We first establish a prime number theorem for $L(s, f_m \otimes \chi)$. Let $\Lambda(n)$ denote the Von Mangoldt function. Note that by \cite[Section 12.3]{iwaniec1997topics}, we have that if we write the logarithmic derivative of $L(s, f_m \otimes \chi)$ as
\begin{align*}
    -\frac{L'(s, f_m\otimes\chi)}{L(s, f_m\otimes\chi)} &= \sum_n a_{m, \chi}(n) \Lambda(n)n^{-s},
\end{align*}
then
\begin{align*}
        a_{m, \chi}(p) &= \begin{cases} 2\chi(p)\cos(m\theta_E(p)) & p \text{ splits in }K \\ 0 & \text{otherwise} \end{cases}
\end{align*}
and $|a_{m, \chi}(p^j)|\le 2$ for all prime powers $p^j, j \geq 2$. Note that when $\chi$ is trivial, we will write $a_{m,\chi} = a_m$. Now, this evidently implies that $\sum_{n \leq x} |a_{m, \chi}(n)\Lambda(n)|^2 \ll x\log^2(x)$. Hence, given the above and Lemma \ref{CMzerofree}, we may apply \cite[Theorem 5.13]{iwaniec2021analytic} and obtain that there exists an absolute constant $\Cr{44} > 0$ for which
\begin{align}\label{CMbound}
    \sum_{n \leq x} a_{m, \chi}(n) \Lambda(n) &\ll x \exp \Big ( \frac{-\Cr{44} \log x}{\sqrt{\log x} + 3\log(N_Eqm)} \Big )\log(xN_Eqm)^4.
\end{align} 

Now, note that
\begin{align*}
    \sum_{\substack{p \leq x \\ p \equiv a \bmod q \\ p \text{ splits in }K}} \cos(m\theta_E(p)) \log p - \frac{1}{2} \sum_{\substack{n \leq x \\ n \equiv a(q)}} a_m(n)\Lambda(n) &=- \frac{1}{2}\sum_{\substack{p \leq x \\ r \geq 2 \\ p^r \equiv a(q)}} a_m(p^r)\log p - \frac{1}{2} \sum_{\substack{p \leq x \\ p \equiv a \bmod q \\ p | N_E}} a_m(p) \log p \\
    &\ll \sqrt{N_E}\log N_E + \sqrt{x}\log x.
\end{align*}
Here we used the fact that by (4), $a_m(p)=2\cos(m\theta_E(p))$ for all $p\nmid N_E$ and $|a_m(p)|\le 2\sqrt{p}$ for all $p$. Hence we have that
\begin{align*}
     \sum_{\substack{p \leq x \\ p \equiv a \bmod q \\ p \text{ splits in }K}} \cos(m\theta_E(p)) \log p  &\ll  \sqrt{N_E}\log N_E+ \sqrt{x} \log x + \sum_{\substack{n \leq x \\ n \equiv a(q)}} a_m(n)\Lambda(n) \\
    &\ll  \sqrt{N_E}\log N_E+ \sqrt{x} \log x + \frac{1}{\varphi(q)} \sum_{\chi(q)} \overline{\chi(a)} \sum_{n \leq x} a_{m, \chi}(n) \Lambda(n)  \\
    &\ll x\exp \Big ( \frac{-\Cr{44}\log x}{\sqrt{\log x} + 3\log(N_Eqm)} \Big )\log (x N_E q m)^4.
\end{align*}
when $x\ge N_E$, as desired. Note that in the summation $\chi$ ranges over all the primitive characters which induce the characters $\bmod \thinspace q$, while to get the last inequality we replace these characters with all the (possibly imprimitive) characters $\bmod \thinspace q$. This generates a negligible error which is absorbed into other error terms.
\end{proof}
For the rest of the paper, the symbol $\chi (q)$ in the subscript for a summation means that the sum ranges through all Dirichlet characters $\chi$ with modulus $q$. Again, we will often replace these characters by the primitive characters that induce them; in every case the error will be negligble.

For the proof of Theorem \ref{master} (2.b), it will also be necessary to derive an estimate for the number of primes in an arithmetic progression which split or remain inert in $K$. Throughout the following proof, let $D_K$ be the discriminant of $K$.

\begin{lem}\label{aqsplit}
There exists absolute constants $\Cl[abcon]{451}, \Cl[abcon]{452}$ such that
\begin{align*}
    (1) \thickspace \thickspace \thickspace &\#\{p \leq x: p \equiv a \bmod q, p \textup{ splits in }K\} \\&= \frac{\pi(x)}{2\varphi(q)} \left(1+\chi_K(a)\delta(K,q)\right) 
    \\& + O \Big( \pi(x) \Big(\frac{x^{-\Cr{451}/\sqrt{q|D_K|}}}{\varphi(q)} + \exp \Big ( \frac{-\Cr{452} \log x}{\sqrt{\log x} + 3 \log (q|D_K|)} \Big ) (\log (xq|D_K|))^4 \Big) \Big). \\
    (2) \thickspace \thickspace \thickspace &\#\{p \leq x: p \equiv a \bmod q, p \textup{ inert in }K\} \\&= \frac{\pi(x)}{2\varphi(q)} \left(1-\chi_K(a)\delta(K,q)\right) 
    \\& + O \Big( \pi(x) \Big(\frac{x^{-\Cr{451}/\sqrt{q|D_K|}}}{\varphi(q)} + \exp \Big ( \frac{-\Cr{452} \log x}{\sqrt{\log x} + 3 \log (q|D_K|)} \Big ) (\log (xq|D_K|))^4 \Big) \Big).
\end{align*}
\end{lem}

\begin{proof} 

The two estimates are proved analogously, so we only prove (1). First, note that \begin{align*}
    \sum_{\substack{p \leq x \\ p \equiv a \bmod q \\ p \text{ splits in }K}} 1 &= \sum_{\substack{p \leq x \\ p \equiv a \bmod q}} \frac{\chi_K(p) + 1}{2} \\
    &= \frac{1}{2}\pi(x; q, a) + \frac{1}{2} \sum_{\substack{p \leq x \\ p \equiv a \bmod q}} \chi_K(p).
\end{align*} Now, given the possibility of a Siegel zero, we have that by \cite[Section 11.3]{montgomery2007multiplicative} and \cite[Theorem 5.28]{iwaniec2021analytic}, there exists absolute constants $\Cr{451}$ and $\Cl[abcon]{453}$ such that $\pi(x; q, a)$ satisfies the following:
\begin{align}\label{arithprogestimate}
     \pi(x; q, a)  = \frac{\Li(x)}{\varphi(q)} + O\Big(  \frac{x}{\log x} \Big ( \frac{x^{-\Cr{451}/\sqrt{q}}}{\varphi(q)} + (\log x) e^{-\Cr{453} \sqrt{\log x}}\Big )\Big).
\end{align} 
where $\Li(x) := \int_2^x \frac{dt}{\ln t}$. Note that by \cite[Theorem 6.14]{montgomery2007multiplicative}, we may replace the $\Li(x)$ in \eqref{arithprogestimate} with $\pi(x)$ and generate an error that is absorbed into the other error terms. Now, consider any primitive character $\chi$ that induces $\chi'\chi_K$ for some $\chi'$ mod $q$. Define $\delta_\chi$ as $1$ if $\chi$ is trivial and $0$ otherwise. Then, by \cite[Theorem 5.13, Theorem 5.28]{iwaniec2021analytic}, there exists an absolute constant $\Cl[abcon]{454}$ such that \begin{align*}
    \sum_{n \leq x} \chi(n) \Lambda(n) &= \delta_\chi x + O \Big ( x^{1 - \Cr{451}/\sqrt{q}} + x\exp \Big ( \frac{-\Cr{454} \log x}{\sqrt{\log x} + 3 \log q} \Big )(\log xq)^4 \Big ).
\end{align*} 
Here we adjust $\Cr{451}$ if necessary to absorb the contribution from $\sqrt{D_K}$, which can only take finitely many values. Also note that by \cite[Theorem 5.28]{iwaniec2021analytic}, the Siegel zero term $x^{1 - \Cr{451}/\sqrt{q}}$ can only exist for at most one of the primitive $\chi$ we consider. Now, note that
\begin{align*}
    \sum_{p \leq x} \chi(p) \log p = \sum_{n \leq x} \chi(n) \Lambda(n) + O(\sqrt{x}\log x).
\end{align*} 
We thus obtain \begin{align}\label{chilogp}
    \sum_{p \leq x} \chi(p) \log p &= \delta_\chi x + O \Big ( x^{1 - \Cr{451}/\sqrt{q}} + x\exp \Big ( \frac{-\Cr{454} \log x}{\sqrt{\log x} + 3 \log q} \Big )(\log xq)^4 \Big ).
\end{align}

We now bound $\sum_{\substack{p \leq x \\ p \equiv a \bmod q}} \chi_K(p) \log p$. We split into two cases. Note that by \cite[Theorem 9.13]{montgomery2007multiplicative}, we have that $\chi_K$ is primitive with conductor $|D_K|$. Hence, when $\delta(K,q)=0$, we have that $\chi'\chi_K$ is nontrivial for all $\chi' \bmod q$ with conductor dividing $q|D_K|$, so thus by (\ref{chilogp}), 
\begin{align*}
    \sum_{\substack{p \leq x \\ p \equiv a \bmod q}} \chi_K(p) \log p &= \frac{1}{\varphi(q)} \sum_{\chi(q)} \overline{\chi(a)} \sum_{p \leq x} \chi(p)\chi_K(p) \log p \\
    &\ll \frac{x^{1 - \Cr{451}/\sqrt{q|D_K|}}}{\varphi(q)} + x \exp \Big ( \frac{-\Cr{454} \log x}{\sqrt{\log x} + 3 \log (q|D_K|)} \Big ) (\log (xq|D_K|))^4.
\end{align*}
When $\delta(K,q)=1$, $\chi'\chi_K$ is trivial for exactly one character $\chi' \bmod q$. For this $\chi'$, we have  $\overline{\chi'(a)} = 1/\overline{\chi_K(a)} = \chi_K(a)$. By (\ref{chilogp}), we thus have
    \begin{align*}
        \sum_{\substack{p \leq x \\ p \equiv a \bmod q}} \chi_K(p) \log p &= \frac{1}{\varphi(q)} \sum_{\chi(q)} \overline{\chi(a)} \sum_{p \leq x} \chi(p)\chi_K(p) \log p \\
        &= \frac{\chi_K(a)}{\varphi(q)}x + O \Big ( \frac{x^{1 - \Cr{451}/\sqrt{q}}}{\varphi(q)} + x \exp \Big ( \frac{-\Cr{454} \log x}{\sqrt{\log x} + 3 \log q} \Big ) (\log xq)^4 \Big ).
    \end{align*}
    We now collate the bounds above. By partial summation, we have
\begin{align*}
    \sum_{\substack{p \leq x \\ p \equiv a \bmod q}} \chi_K(p) &= \frac{\pi(x)}{\varphi(q)}\chi_K(a)\delta(K,q) \\&+O\bigg(\frac{1}{\log x} \Big ( \frac{x^{1 - \Cr{451}/\sqrt{q|D_K|}}}{\varphi(q)} + x \exp \Big ( \frac{-\Cr{454} \log x}{\sqrt{\log x} + 3 \log (q|D_K|)} \Big ) (\log (xq|D_K|))^4 \Big )\bigg).
\end{align*} Hence, there exists an absolute constant $\Cr{452}$ such that
\begin{align*}
    \sum_{\substack{p \leq x \\ p \equiv a \bmod q \\ p \text{ splits in }K}} 1 &= \frac{1}{2} \pi(x; q, a) + \frac{1}{2} \sum_{\substack{p \leq x \\ p \equiv a \bmod q}} \chi_K(p) \\&= \frac{\pi(x)}{2\varphi(q)} \left(1+\chi_K(a)\delta(K,q)\right) \\&+ O \Big( \pi(x) \Big(\frac{x^{-\Cr{451}/\sqrt{q|D_K|}}}{\varphi(q)} + \exp \Big ( \frac{-\Cr{452} \log x}{\sqrt{\log x} + 3 \log (q|D_K|)} \Big ) (\log (xq|D_K|))^4 \Big) \Big).
\end{align*}
\end{proof}

\subsection{Rankin-Selberg convolutions of a CM and Non-CM newform}

Consider two elliptic curves $E, E'$, the former having complex multiplication over an imaginary quadratic field $K$, and the latter without complex multiplication over any imaginary quadratic field. Let $\xi$ denote the primitive Gr\"{o}ssencharacter which satisfies $L(s, E) = L(s, \xi)$, $g_m$ (instead of $f_m$) denote the cusp form which induces $\xi_m$, and $\pi_{g_m} \in \mathfrak{F}_2$ denote the representation corresponding to $g_m$. Also let $f'$ denote the cusp form which corresponds to $E'$, and let $\pi_{f'} \in \mathfrak{F}_2$ be the representation which corresponds to $f'$.

\begin{lem}\label{CMnonCMPNT}
There exists an absolute constant $\Cl[abcon]{161} > 0$ such that if $M=\frac{\Cr{161}\sqrt{\log x}}{\log(N_{E}N_{E'}\log x)}$, then for all $1 \leq m, m' \leq M$,
\begin{align*}
    &\sum_{\substack{p\le x\\ p\nmid N_{E}N_{E'}}}\cos(m\theta_{E}(p))U_m(\cos \theta_{E'}(p)) \log p \\&\ll m'^2x\Big(x^{- \frac{1}{32\Cr{43}M^2}} + \exp\Big(-\frac{\Cr{461}\log x}{4M^2\log(N_{E}{N_{E'}}M)}\Big)+\exp\Big(-\frac{\sqrt{\Cr{461}\log x}}{2M}\Big) \Big).
\end{align*}
\end{lem}
\begin{proof}
    By the information in Sections \ref{secinfo} and \ref{ssympowerinfo}, along with \cite[Equation 5.86]{iwaniec2021analytic} and \cite[Equation 3.3]{thorner2021effective}, we have that the Langlands parameters $\mu_{\pi_{g_m} \times \Sym^{m'} \pi_{f'}}$ of $L(s, \pi_{g_m} \times \Sym^{m'} \pi_{f'})$ satisfy that $\mu_{\pi_{g_m} \times \Sym^{m'} \pi_{f'}} = 0$ or $\textup{Re}(\mu_{\pi_{g_m} \times \Sym^{m'} \pi_{f'}}) \geq \frac{1}{2}$. Combining this with (\ref{GRC}), we find that there exists an absolute constant $\Cr{161} > 0$ such that if $M$ is defined as above, then $\pi_{g_m}\times \text{Sym}^{m'}\pi_{f'}$ satisfies the conditions of \cite[Proposition 5.1]{thorner2021effective} for all $1\le m,m'\le M$. Applying the Proposition, we find that
\begin{align*}
    &\sum_{\substack{p\le x\\ p\nmid N_{E}N_{E'}}}\cos(m\theta_{E}(p))U_m(\cos \theta_{E'}(p)) \log p\\&\ll m'^2x^{1 - \frac{1}{32\Cr{43}M^2}} + m'^2x\Big(\exp\Big(-\frac{\Cr{461}\log x}{4M^2\log(N_{E}{N_{E'}}M)}\Big)+\exp\Big(-\frac{\sqrt{\Cr{461}\log x}}{2M}\Big) \Big),
\end{align*}
as desired.
\end{proof}

\subsection{Rankin-Selberg convolutions of Two CM newforms}

Consider two twist-inequivalent elliptic curves $E, E'$ having complex mulitplication over two imaginary quadratic fields $K, K'$ (respectively). Let $m, m' \geq 1$ be integers. Consider the Gr\"{o}ssencharacters $\xi, \xi'$ corresponding to $E, E'$, let $f = f_m, f' = f'_{m'}$ be the cusp forms corresponding to $\xi_m, \xi'_{m'}$ (as defined in \ref{secinfo}), and let $\pi_f, \pi_{f'} \in \mathfrak{F}_2$ denote the representations corresponding to $f, f'$ respectively. Finally, let $\mathfrak{q} = N_EN_{E'}mm'$.

\begin{lem}\label{coscoslogp}

There exists an absolute constant $\Cl[abcon]{471} > 0$ such that 
$$
\sum_{\substack{p \leq x \\ p \textup{ splits in }K}}\cos(m\theta_{E}(p))\cos(m'\theta_{E'}(p)) \log p \ll x\exp\bigg(\frac{-\Cr{471}\log x}{\sqrt{\log x}+3\log \mathfrak{q}}\bigg)(\log x\mathfrak{q})^4.
$$

\end{lem}

\begin{proof}

This is trivial for $x< \mathfrak{q}$; thus assume $x\ge \mathfrak{q}$. Let $\Lambda(n)$ denote the Von Mangoldt function, and let the Dirichlet series expansion of $L(s, \pi_f \times \pi_{f'})$ be $\sum_{n \geq 1} a_{\pi_f \times \pi_{f'}}(n)n^{-s}$. Define
\begin{align*}
    \psi(x, \pi_f \times \pi_{f'}) &= \sum_{n\le x}a_{\pi_f \times \pi_{f'}}(n) \Lambda(n).
\end{align*}
Note that by definition $a_{\pi_f \times \pi_{f'}}(p) = \cos(m\theta_{E}(p))\cos(m'\theta_{E'}(p))$ when $p \nmid N_{E}N_{E'}$, and moreover this quantity equals $0$ when $p$ is inert in either $K$ or $K'$. Hence
\begin{align*}
    \psi(x, \pi_f \times \pi_{f'}) &= \sum_{\substack{p \leq x \\ p \nmid N_{E}N_{E'} \\ p \text{ splits in }K\textup{~and~} K'}} \cos(m\theta_{E}(p))\cos(m\theta_{E'}(p))\log p + O\bigg(\sum_{p^j\le x, j\ge 2}\log p+\sum_{p|N_{E}N_{E'}}\sqrt{p}\log p \bigg)
    \\&= \sum_{\substack{p \leq x \\ p \nmid N_{E}N_{E'} \\ p \textup{ splits in }K \textup{~and~} K'}} \cos(m\theta_{E}(p))\cos(m\theta_{E'}(p))\log p+O\left(\sqrt{x}\log x+\sqrt{N_{E}N_{E'}}\log(N_{E}N_{E'})\right).
\end{align*}
Now, note that \cite[Equation 5.48]{iwaniec2021analytic} holds trivially for $L(s, \pi_f \times \pi_{f'})$. Hence, given the information above along with Lemma \ref{CMCMsiegelzero}, we may apply \cite[Theorem 5.13]{iwaniec2021analytic} (substituting \cite[Equation 5.39]{iwaniec2021analytic} with ($\ref{CMCMsiegelzero}$)) and get that there exists an absolute constant $\Cr{471} > 0$ such that
$$
\psi(x, \pi_f \times \pi_{f'}) \ll x\exp\bigg(-\frac{\Cr{471}\log x}{\sqrt{\log x}+3\log \mathfrak{q}}\bigg)(\log x\mathfrak{q})^4,
$$
Substituting in $\psi$, we now get our desired result.
\end{proof}

\section{Proof of Theorem \ref{master} (2.a)}\label{2a}

Fix the notation from Sections \ref{sintro}, \ref{sinfo}, and \ref{snonCMzerofree}. In particular, let $a, q$ be coprime positive integers, and consider an elliptic curve $E$ without complex multiplication over any imaginary quadratic field.

\begin{proof}[Proof of Theorem \ref{master} (2.a)]

We prove the Theorem for $x \geq N_E$; for $x < N_E$ the result is trivial. We start with some preliminary manipulations. Let $U_m$ denote the $m$th Chebyshev polynomial of the second kind. By Lemma \ref{bspoly1}, we have
\begin{align*}
    \pi_{I}(x; q,a) &= \sum_{\substack{p\le x\\ p\equiv a (q)}} \one_I(\theta_E(p))
    \\&\leq \sum_{\substack{p\le x\\ p\nmid N_E\\ p\equiv a (q)}} \Bigg( \mu_{\text{ST}}(I)+\Cr{bspoly0}\bigg(\frac{1}{M}+\sum_{m=1}^{M}\bigg(\frac{1}{m}U_m(\cos \theta_E(p))\bigg)\bigg) \Bigg)
    \\&= \mu_{\text{ST}}(I)\pi(x; q,a)+O\Bigg(\frac{\pi(x; q,a)}{M}+\sum_{\substack{p\le x\\ p\equiv a (q)}}\sum_{m=1}^{M}\bigg(\frac{1}{m}U_m(\cos \theta_E(p))\bigg)\Bigg),
\end{align*}
We may similarly bound $\pi_I(x;q, a)$ from below to obtain that
\begin{align}\label{eq41}
    \pi_I(x; q, a) &= \mu_{\text{ST}}(I)\pi(x; q,a)+O\Bigg(\frac{\pi(x; q,a)}{M}+\sum_{\substack{p\le x\\ p\equiv a (q)}}\sum_{m=1}^{M}\bigg(\frac{1}{m}U_m(\cos \theta_E(p))\bigg)\Bigg).
\end{align}
Moreover, note that \begin{align*}
    \sum_{\substack{p\le x\\ p\equiv a (q)}}\sum_{m=1}^{M}\frac{1}{m}U_m(\cos \theta_E(p)) &= \frac{1}{\varphi(q)}\sum_{\chi \bmod q}\overline{\chi(a)} \Bigg( \sum_{m=1}^{M}\frac{1}{m} \bigg( \sum_{p\le x}\chi(p)U_m(\cos \theta_E(p)) \bigg) \Bigg).
\end{align*} 

Now, by Lemma \ref{nonCMPNT}, we have that there exists a sufficiently small absolute constant $\Cr{2aM}$ such that if $M=\frac{\Cr{2aM}\sqrt{\log x}}{\log(N_Eq\log x))}$, then for all $1 \leq m \leq M,$
\begin{align*}
    \sum_{\substack{p\le x\\ p\nmid N_E}}\chi(p)U_m(\cos \theta_E(p)) \log p \ll m^2x \Big( x^{-\frac{1}{\Cr{43}m}}+\Big(\exp\Big(-\frac{\Cr{42}\log x}{2m^2\log(N_Eqm)}\Big)+\exp\Big(-\frac{\Cr{42}\sqrt{\log x}}{2\sqrt{m}} \Big)\Big) \Big).
\end{align*}
Applying partial summation and inserting the contributions from primes dividing $N_E$ (which is negligble for $x > N_E$), we obtain
\begin{align*}
    \sum_{p \leq x}\chi(p)U_m(\cos \theta_E(p)) \ll \frac{m^2x}{\log x} \Big( x^{-\frac{1}{\Cr{43}m}}+\Big(\exp\Big(-\frac{\Cr{42}\log x}{2m^2\log(N_Eqm)}\Big)+\exp\Big(-\frac{\Cr{42}\sqrt{\log x}}{2\sqrt{m}} \Big)\Big) \Big).
\end{align*}
 Substituting this into (\ref{eq41}), we obtain that \begin{align*}
    &\pi_{I}(x; q,a)-\mu_{\text{ST}}(I)\pi(x; q,a)-\frac{\pi(x)}{\varphi(q)M}\\
    &\ll \frac{1}{\varphi(q)}\sum_{\chi \bmod q}\overline{\chi(a)} \sum_{m=1}^M \bigg( \frac{mx}{\log x} \Big( x^{-\frac{1}{\Cr{43}m}}+\Big(\exp\Big(-\frac{\Cr{42}\log x}{2m^2\log(N_Eqm)}\Big)+\exp\Big(-\frac{\Cr{42}\sqrt{\log x}}{2\sqrt{m}} \Big)\Big) \Big) \bigg) \\
    &\ll \frac{M^2x}{\log x}\Big(x^{-\frac{1}{\Cr{43}M}}+\exp\Big(-\frac{\Cr{42}\log x}{2M^2\log(N_{E}qM)}\Big)+\exp\Big(-\frac{\Cr{42}\sqrt{\log x}}{2\sqrt{M}}\Big)\Big).
\end{align*}
Adjusting $\Cr{2aM}$ to be suitably small, we find that for all $x \geq 2$,
\begin{align*}
    &\pi(x)\bigg(\frac{1}{\varphi(q)M}+M^2\Big(x^{-\frac{1}{\Cr{43}M}}+\exp\Big(-\frac{\Cr{42}\log x}{2M^2\log(N_{E}qM)}\Big)+\exp\Big(-\frac{\Cr{42}\sqrt{\log x}}{2\sqrt{M}}\Big)\Big)\bigg) \\&\ll \pi(x; q,a)\frac{\log(N_{E}q\log x)}{\sqrt{\log x}},
\end{align*} Finally, recall that by (\ref{arithprogestimate}),
\begin{align*}
    \pi(x; q, a) = \frac{\pi(x)}{\varphi(q)}+O\bigg(\frac{x}{\log x} \Big ( \frac{x^{-\Cr{451}/\sqrt{q}}}{\varphi(q)} + (\log x) e^{-\Cr{453} \sqrt{\log x}}\Big )\bigg).
\end{align*}
Collating the above bounds, we get our desired estimate.
\end{proof}

\section{Proof of Theorem \ref{master} (2.b)}\label{2b}
Throughout this section, we will fix the notation from Sections \ref{sintro}, \ref{sinfo}, and \ref{sCMzerofree}. In particular, we will consider an elliptic curve $E$ which has complex multiplication over an imaginary quadratic field $K$, with its theta values $\theta_E(p)$ corresponding to its traces of Frobenius modulo $p$. 

\begin{proof}[Proof of Theorem \ref{master} (2.b)]
The bound is trivially true in the case $x < N_Eq$, so we only prove the case $x \geq N_Eq$. Note that
\begin{align}\label{cmap}
    \pi_{I}(x; q,a)&=\sum_{\substack{p \leq x \\ p \equiv a \bmod q \\ p \textup{ splits in }K}} \one_I(\theta_E(p)) + \one_{\pi/2 \in I} \sum_{\substack{p \leq x \\ p \equiv a \bmod q \\ p \textup{ inert}}} 1.
\end{align}
We first estimate the first term on the right side of (\ref{cmap}). By Lemma \ref{bspoly1}, we have that
\begin{align*}
    \sum_{\substack{p \leq x \\ p \equiv a \bmod q \\ p \textup{ splits in }K}} \one_I(\theta_E(p)) &\leq \sum_{\substack{p \leq x \\ p \equiv a \bmod q \\ p \textup{ splits in }K}} \bigg(\frac{|I|}{\pi} + \Cr{bspoly0}\bigg(\frac{1}{M}+\sum_{1 \leq m \leq M} \frac{\cos(m\theta_E(p))}{m} \bigg)\bigg) \nonumber \\
    &= \bigg(\sum_{\substack{p \leq x \\ p \equiv a \bmod q \\ p \text{ splits in }K}} \frac{|I|}{\pi} + O\Big(\frac{1}{M}\Big)\bigg) + O\bigg(\sum_{1 \leq m \leq M} \frac{1}{m} \sum_{\substack{p \leq x \\ p \equiv a \bmod q \\ p \text{ splits in }K}} \cos(m\theta_E(p))\bigg).
\end{align*}
We may similarly bound the left hand side from below, to obtain that
\begin{align}
    \sum_{\substack{p \leq x \\ p \equiv a \bmod q \\ p \textup{ splits in }K}} \one_I(\theta_E(p)) &= \bigg(\sum_{\substack{p \leq x \\ p \equiv a \bmod q \\ p \text{ splits in }K}} \frac{|I|}{\pi} + O\Big(\frac{1}{M}\Big)\bigg) + O\bigg(\sum_{1 \leq m \leq M} \frac{1}{m} \sum_{\substack{p \leq x \\ p \equiv a \bmod q \\ p \text{ splits in }K}} \cos(m\theta_E(p))\bigg).\label{cmap1}
\end{align}
Now, using Lemma \ref{coslogp} and partial summation, we find that \begin{align}
        \sum_{\substack{p \leq x \\ p \equiv a \bmod q \\ p \text{ splits in }K}} \cos(m\theta_E(p)) &\ll \frac{1}{\log x} \Big ( x\exp \Big ( \frac{-\Cr{44}\log x}{\sqrt{\log x} + 3\log(N_Eqm)} \Big )(\log (x N_E q m)^4 \Big ) \nonumber \\
        &+ \Big( \int_{1}^{x}(\log t)^2\exp \Big ( \frac{-\Cr{44}\log t}{\sqrt{\log t} + 3\log(N_Eqm)} \Big )\log (N_E q m)^4 dt \Big)   \nonumber     \\
        &\ll \frac{1}{\log x} \Big ( x\exp \Big ( \frac{-\Cr{44}\log x}{\sqrt{\log x} + 3\log(N_Eqm)} \Big )(\log (x N_E q m)^4 \Big ). \label{cmap2}
\end{align} 
Here we use that for $x\ge N_Eq$, \begin{align*}
    &\int_{1}^{x}(\log t)^2\exp \Big ( \frac{-\Cr{44}\log t}{\sqrt{\log t} + 3\log(N_Eqm)} \Big )\log (N_E q m)^4 dt 
    \\&= \int_{0}^{\log x}\exp\Big(t- \frac{-\Cr{44}t}{\sqrt{t} + 3\log(N_Eqm)} \Big )t^2\log (N_E q m)^4 dt \\&\ll x(\log x)^3\exp \Big ( \frac{-\Cr{44}\log x}{\sqrt{\log x} + 3\log(N_Eqm)} \Big )\log (N_E q m)^4.
\end{align*}
We now use (\ref{cmap2}) and Lemma \ref{aqsplit} to estimate the error in (\ref{cmap1}). We first investigate the contribution from the $p$ which split in $K$. Recall that
\begin{align*}
    \sum_{\substack{p \leq x \\ p \equiv a \bmod q \\ p \text{ splits in }K}} \one_I(\theta_E(p)) - \sum_{\substack{p \leq x \\ p \equiv a \bmod q \\ p \text{ splits in }K}} \frac{|I|}{\pi} &\ll \sum_{\substack{p \leq x \\ p \equiv a \bmod q \\ p \text{ splits in }K}} \frac{1}{M} + \sum_{1 \leq m \leq M} \sum_{\substack{p \leq x \\ p \equiv a \bmod q \\ p \text{ splits in }K}} \frac{1}{m} \cos(m\theta_E(p)).
\end{align*}
Hence by Lemma \ref{aqsplit} (1), we have
\begin{align*}
    &\sum_{\substack{p \leq x \\ p \equiv a \bmod q \\ p \text{ splits in }K}} \one_I(\theta_E(p)) - \frac{|I|}{\pi} \frac{\pi(x)}{2\varphi(q)} \Big ( 1 + \chi_K(a)\delta(K,q) \Big ) \\\nonumber&\ll \frac{x}{\log x} \Big( \frac{x^{-\Cr{451}/\sqrt{q}}}{\varphi(q)} + \exp \Big ( \frac{-\Cr{452} \log x}{\sqrt{\log x} + 3 \log q} \Big ) (\log xq)^4 \Big) + \frac{\pi(x)}{M\varphi(q)} \\\nonumber &+ \sum_{1 \leq m \leq M} \frac{x}{m\log x}  \Big ( \exp \Big ( \frac{-\Cr{44}\log x}{\sqrt{\log x} + 3\log(N_Eqm)} \Big )(\log (x N_E q m)^4 \Big ).
\end{align*}
Note that since there are finitely many possibilities for $K$, we may drop the $|D_K|$ contribution to the error term. Similarly, by Lemma \ref{aqsplit} (2) we may estimate the contribution from the $p$ which remain inert in $K$ as
\begin{align*}
    \one_{\pi/2 \in I}\sum_{\substack{p \leq x \\ p \equiv a \bmod q \\ p \text{ inert in }K}} 1
    &= \one_{\pi/2 \in I} \Bigg(\frac{\pi(x)}{2\varphi(q)} \left(1-\chi_K(a)\delta(K,q)\right) \\&+ O \Big( \pi(x) \Big(\frac{x^{-\Cr{451}/\sqrt{q}}}{\varphi(q)} + \exp \Big ( \frac{-\Cr{452} \log x}{\sqrt{\log x} + 3 \log q} \Big ) (\log xq)^4 \Big) \Big) \Bigg).
\end{align*}
Now, collating the bounds in the previous two equations gives
\begin{align*}
    & \pi_I(x;q,a) - \frac{|I|}{\pi} \frac{\pi(x)}{2\varphi(q)} \Big ( 1 + \chi_K(a)\delta(K,q) \Big ) - \one_{\pi/2 \in I} \Big( \frac{\pi(x)}{2\varphi(q)} \Big ( 1 - \chi_K(a)\delta(K,q) \Big ) \Big) \\
    &\ll\frac{x}{\log x} \Big( \frac{x^{-\Cr{451}/\sqrt{q}}}{\varphi(q)} + \exp \Big ( \frac{-\Cr{452} \log x}{\sqrt{\log x} + 3 \log q} \Big ) (\log xq)^4 \Big) + \frac{\pi(x)}{M\varphi(q)} \\ &+ \frac{x\log M}{\log x}  \Big ( \exp \Big ( \frac{-\Cr{44}\log x}{\sqrt{\log x} + 3\log(N_EqM)} \Big )(\log (x N_E q M)^4 \Big ).
\end{align*}
Setting $M = \exp(\sqrt{\log x})$ now gives us that there exists an absolute constant $\Cl[abcon]{457}$ such that
\begin{align*}
    &\bigg|\pi_I(x;q,a) - \frac{\pi(x)}{2\varphi(q)}\bigg(\frac{|I|}{\pi} + \one_{\pi/2 \in I} + \chi_K(a)\delta(K, q)\Big(\frac{|I|}{\pi} - \one_{\pi/2 \in I}\Big)\bigg)\bigg| \\
    &\ll \Big( \pi(x) \Big( \frac{x^{-\Cr{451}/\sqrt{q}}}{\varphi(q)} + (\log x)^{9/2} \exp \Big ( \frac{-\Cr{457}\log x}{\sqrt{\log x} + \log(N_Eq)} \Big ) \Big) \Big),
\end{align*}
as desired.
\end{proof}

\section{Proof of Theorem \ref{master} (1.b)}\label{1b}

Throughout this section, we will fix the notation from Sections \ref{sintro}, \ref{sinfo}, and \ref{sCMnonCMzerofree}. In particular, we will consider two elliptic curves $E, E'$, the former having complex multiplication over an imaginary quadratic field $K$, the latter without complex multiplication over any imaginary quadratic field. Moreover, let $\xi$ denote the primitive Gr\"{o}ssencharacter which satisfies $L(s, E) = L(s, \xi)$, and let $g_m$ denote the cusp form which induces $\xi_m$ (where $\xi_m$ is as defined in Section \ref{secinfo}).

\begin{proof}[Proof of Theorem \ref{master} (1.b)]
We have 
\begin{align}\label{section60}
    \pi_{I,I'}(x) = \sum_{\substack{p\le x\\ p \text{ splits in }K}} \one_{I}(\theta_{E}(p))\one_{I'}(\theta_{E'}(p))+\one_{\frac{\pi}{2}\in I}\sum_{\substack{p\le x\\ p \text{ inert in }K}}\one_{I'}(\theta_{E'}(p)).
\end{align}
We first estimate the contribution from the splitting primes to (\ref{section60}). Applying Lemma \ref{bspoly2}, we obtain \begin{align}  \label{cmnoncmerror}
    \bigg|\sum_{\substack{p\le x\\ p \text{ splits}}} \one_{I}(\theta_{E}(p))\one_{I'}(\theta_{E'}(p)) - \frac{|I|}{\pi}\mu_{\text{ST}}(I')\sum_{\substack{p \leq x \\ p \text{ splits}}}&1 \bigg| \\
    \ll \nonumber \frac{\pi(x)}{M}+\sum_{\substack{p\le x\\ p \text{ splits}}} \bigg(\sum_{m=1}^{M}\frac{\cos(m\theta_{E}(p))+U_m(\cos\theta_{E'}(p))}{m} &+ \sum_{1\le m, m'\le M}\frac{\cos(m\theta_{E}(p))U_{m'}(\cos\theta_{E'}(p))}{mm'} \bigg).
\end{align}

We first bound the third term on the right side of (\ref{cmnoncmerror}). By Lemma \ref{CMnonCMPNT}, there exists an absolute constant $\Cr{161} > 0$ such that if $M=\frac{\Cr{161}\sqrt{\log x}}{\log(N_{E}N_{E'}\log x)}$, then for all $1 \leq m, m' \leq M$,
\begin{align*}
    &\sum_{\substack{p\le x\\ p\nmid N_{E}N_{E'}}}\cos(m\theta_{E}(p))U_m(\cos \theta_{E'}(p)) \log p\\&\ll m'^2x\Big(x^{- \frac{1}{32\Cr{161}M^2}} + \exp\Big(-\frac{\Cr{461}\log x}{4M^2\log(N_{E}{N_{E'}}M)}\Big)+\exp\Big(-\frac{\sqrt{\Cr{461}\log x}}{2M}\Big) \Big).
\end{align*}
Now, by partial summation, we have \begin{align*}
    &\sum_{p \leq x}\sum_{1\le m, m'\le M}\frac{\cos(m\theta_{E}(p))U_{m'}(\cos\theta_{E'}(p))}{mm'} \\&\ll M^2\log M\Big(\frac{x}{\log x}\Big)\Big(x^{- \frac{1}{32\Cr{161}M^2}} + \exp\Big(-\frac{\Cr{461}\log x}{4M^2\log(N_{E}{N_{E'}}M)}\Big)+\exp\Big(-\frac{\sqrt{\Cr{461}\log x}}{2M}\Big) \Big).
\end{align*}
Hence, we may estimate the first and third terms on the right side of (\ref{cmnoncmerror}) as follows: \begin{align}\label{cmnoncmerror1}
    &\frac{\pi(x)}{M}+\sum_{\substack{p\le x\\ p \text{ splits}}}\sum_{1\le m, m'\le M}\frac{\cos(m\theta_{E}(p))U_{m'}(\cos\theta_{E'}(p))}{mm'}\\\nonumber&\ll \pi(x) \bigg(\frac{1}{M}+M^2\log M\Big(x^{- \frac{1}{32\Cr{43}M^2}} + \exp\Big(-\frac{\Cr{461}\log x}{4M^2\log(N_{E}{N_{E'}}M)}\Big)+\exp\Big(-\frac{\sqrt{\Cr{461}\log x}}{2M}\Big)\Big)\bigg)
    \\\nonumber&\ll \pi(x)\frac{\log(N_{E}N_{E'}\log x)}{\sqrt{\log x}}.
\end{align}
We now bound the second term on the right side of (\ref{cmnoncmerror}). Note that
\begin{align*}
    \sum_{\substack{p\le x\\ p \text{ splits}}} \sum_{m=1}^{M}\frac{\cos(m\theta_{E}(p))+U_m(\cos\theta_{E'}(p))}{m} &= \sum_{\substack{p\le x}} \bigg(\frac{1 + \chi_K(p)}{2}\bigg)\sum_{m=1}^{M}\frac{\cos(m\theta_{E}(p))+U_m(\cos\theta_{E'}(p))}{m}.
\end{align*}
Now, by Lemma \ref{nonCMPNT} and partial summation, we have that
\begin{align*}
    \sum_{\substack{p\le x}} \bigg(\frac{1 + \chi_K(p)}{2}\bigg)U_m(\cos\theta_{E'}(p)) &\ll \frac{m^2x}{\log x} \Big(\exp\Big(-\frac{\Cr{42}\log x}{2m^2\log(N_Em)}\Big)+\exp\Big(-\frac{\Cr{42}\sqrt{\log x}}{2\sqrt{m}} \Big) \Big).
\end{align*}
Note that since the possibilities for $\chi_K$ are finite, we may omit both the contribution from $q$ and the possible Siegel zero. Moreover, by using (\ref{CMbound}) and applying partial summation, we may obtain
\begin{align*}
    \sum_{\substack{p\le x}} \bigg(\frac{1 + \chi_K(p)}{2}\bigg)\cos(m\theta_{E}(p)) &\ll \frac{x}{\log x} \exp \Big ( \frac{-\Cr{44} \log x}{\sqrt{\log x} + 3\log(N_Em)} \Big )\log(xN_Em)^4.
\end{align*}
Note that the contribution from the prime powers to (\ref{CMbound}) can be absorbed into the other error terms (by the same reasoning as in the proof to Lemma \ref{coslogp}), and hence has been neglected above. Now, collating the above bounds and substituting in the value of $M$, we have that
\begin{align}\label{section62}
    &\sum_{\substack{p\le x\\ p \text{ splits}}} \sum_{m=1}^{M}\frac{\cos(m\theta_{E}(p))+U_m(\cos\theta_{E'}(p))}{m} \\
    \nonumber &\ll \sum_{m=1}^M \frac{mx}{\log x} \Big( \Big(\exp\Big(-\frac{\Cr{42}\log x}{2m^2\log(N_Em)}\Big)+\exp\Big(-\frac{\Cr{42}\sqrt{\log x}}{2\sqrt{m}} \Big)\Big) \Big) 
    \\ \nonumber&+ \frac{x}{m\log x} \exp \Big ( \frac{-\Cr{44} \log x}{\sqrt{\log x} + 3\log(N_Em)} \Big )\log(xN_Em)^4 \\
    \nonumber &\ll \frac{M^2x}{\log x} \Big( \Big(\exp\Big(-\frac{\Cr{42}\log x}{2M^2\log(N_EM)}\Big)+\exp\Big(-\frac{\Cr{42}\sqrt{\log x}}{2\sqrt{M}} \Big)\Big) \Big) 
    \\ \nonumber &+ \frac{x\log M}{\log x} \exp \Big ( \frac{-\Cr{44} \log x}{\sqrt{\log x} + 3\log(N_EM)} \Big )\log(xN_EM)^4 \\
    \nonumber &\ll \pi(x)\frac{\log(N_{E})^4\log(N_{E'}\log x)}{\sqrt{\log x}}.
\end{align}

We now estimate the contribution from inert primes to (\ref{section60}). However, note that if $D_K$ is the conductor of $\chi_K$, then exactly half of the residues in $(\mathbb{Z}/D_K\mathbb{Z})^\times$ correspond to inert primes. Hence, by Theorem \ref{master} (2.a) we have \begin{align}\label{cmnoncmerror2}
    \one_{\frac{\pi}{2}\in I}\sum_{\substack{p\le x\\ p \text{ inert}}}\one_{I'}(\theta_{E'}(p))&= \one_{\frac{\pi}{2}\in I} \Bigg( \frac{1}{2}\mu_{\text{ST}}(I')\pi(x)+O\bigg(\pi(x)\frac{\log(N_{E'}\log x)}{\sqrt{\log x}}\bigg) \Bigg).
\end{align} Combining (\ref{cmnoncmerror1}), (\ref{section62}), and (\ref{cmnoncmerror2}) now concludes the proof. \end{proof}

\section{Proof of Theorem \ref{master} (1.c)}\label{1c}

Throughout this section, we will fix the notation from Sections \ref{sintro}, \ref{sinfo}, and \ref{sCMCMzerofree}. In particular, we will consider two elliptic curves $E, E'$ having complex multiplication over two imaginary quadratic fields $K, K'$ (respectively). 

\begin{proof}[Proof of Theorem \ref{master} (1.c)]

The proof is trivial when $x < N_EN_{E'}$; hence assume $x \geq N_EN_{E'}$. First, recall that by Lemmas \ref{coslogp} and \ref{coscoslogp} and applying partial summation, we have that for all positive integers $m, m'$, \begin{align}\label{sumcos}
    \sum_{\substack{p \leq x \\ p \text{ splits in }K}} \cos(m\theta_p) \ll \bigg ( \frac{x}{\log x} \bigg )\exp \bigg ( \frac{-\Cr{44}\log x}{\sqrt{\log x} + 3\log(N_Em)} \bigg )\log (x N_E m)^4
\end{align} and \begin{align}\label{sumcoscos}
    \sum_{\substack{p \leq x \\ p \text{ splits in }K}}\cos(m\theta_{E}(p))\cos(m'\theta_{E'}(p)) \ll \bigg(\frac{x}{\log x}\bigg )\exp\bigg(\frac{-\Cr{471}\log x}{\sqrt{\log x}+3\log \mathfrak{q}}\bigg)(\log x\mathfrak{q})^4\bigg ).
\end{align}
Now we consider \begin{align}
    \nonumber \pi_{I,I'}(x) & = \sum_{p\le x}\mathbf{1}_{I}(\theta_{E}(p))\mathbf{1}_{I'}(\theta_{E'}(p))
    \\ \label{4terms} &= \sum_{\substack{p \leq x \\ p \text{ splits in } K \\ p \text{ splits in } K'}}\mathbf{1}_{I}(\theta_{E}(p))\mathbf{1}_{I'}(\theta_{E'}(p))+\mathbf{1}_{\frac{\pi}{2}\in I'}\sum_{\substack{p \leq x \\ p \text{ splits in } K\\ p \text{ inert in } K'}}\mathbf{1}_{I}(\theta_{E}(p))
    \\& \nonumber +\mathbf{1}_{\frac{\pi}{2}\in I}\sum_{\substack{p \leq x \\ p \text{ inert in } K\\ p \text{ splits in } K'}}\mathbf{1}_{I'}(\theta_{E'}(p))+\mathbf{1}_{\frac{\pi}{2}\in I}\mathbf{1}_{\frac{\pi}{2}\in I'}\sum_{\substack{p \leq x \\ p \text{ inert in } K\\ p \text{ inert in } K'}}1 + O(N_EN_{E'}).
\end{align}
Here $O(N_EN_{E'})$ comes from the contribution of the omitted primes $p$ which divide $N_{E}N_{E'}$. Now, (\ref{4terms}) splits $\pi_{I,I'}(x)$ into 4 terms. We now simplify the first $3$ terms.

\textbf{Term 1.} By Lemma \ref{bspoly2}, we have \begin{align} \label{term1}
    \sum_{\substack{p \leq x \\ p \text{ splits in } K\\ p \text{ splits in } K'}}\mathbf{1}_{I}(\theta_{E}(p))\mathbf{1}_{I'}(\theta_{E'}(p))
\end{align} is equal to \begin{align*}
    \frac{|I||I'|}{\pi^2}\sum_{\substack{p \leq x \\ p \text{ splits in } K\\ p \text{ splits in } K'}}1&+O\bigg(\frac{\pi(x)}{M}+\sum_{m=1}^{M}\frac{1}{m}\bigg(\cos(m\theta_{E}(p))+\cos(m'\theta_{E'}(p))\bigg) \\\nonumber&+\sum_{1\le m, m'\le M}\frac{1}{mm'}\cos(m\theta_{E}(p))\cos(m'\theta_{E'}(p))\bigg).
\end{align*}
By (\ref{sumcos}) and (\ref{sumcoscos}), the error term above is bounded by \begin{align} \label{1248}
    &O\bigg(\frac{x}{\log x} \bigg(\frac{1}{M}+ \log M \exp \bigg ( \frac{-\Cr{44}\log x}{\sqrt{\log x} + 3\log N_{E}N_{E'}M} \bigg )\log (x N_{E}N_{E'} M)^4 \\\nonumber&+(\log M)^2\exp\bigg(\frac{-\Cr{471}\log x}{\sqrt{\log x}+3\log N_{E}N_{E'}M}\bigg)\log (xN_{E}N_{E'}M)^4\bigg)\bigg).
\end{align} Hence, there exists an absolute constant $\Cl[abcon]{473}$ such that the term in (\ref{1248}) is bounded by \begin{align*}
    O\bigg(\frac{x}{\log x} \bigg( \frac{1}{M}+(\log M)^2\exp\bigg(\frac{-\Cr{473}\log x}{\sqrt{\log x}+3\log N_{E}N_{E'}M}\bigg)\log (xN_{E}N_{E'}M)^4 \bigg) \bigg).
\end{align*}
Taking $M=\exp(\sqrt{\log x})$, we obtain that there exists an absolute constant $\Cl[abcon]{474}$ such that (\ref{term1}) is equal to \begin{align*}
    \frac{|I||I'|}{\pi^2}\bigg(\sum_{\substack{p \leq x \\ p \text{ splits in } K\\ p \text{ splits in } K'}}1\bigg)+O\bigg(\frac{x}{\log x} \bigg(\exp\bigg(\frac{-\Cr{474}\log x}{\sqrt{\log x}+3\log N_{E}N_{E'}} \bigg)\bigg)\bigg).
\end{align*}

\textbf{Term 2.} When $\frac{\pi}{2}$ is not in $I'$ or $K \sim K'$, the second term is automatically 0. Hence, assume $\frac{\pi}{2}\in I'$ and $K\not \sim K'$. In this case, \begin{align*}
    \sum_{\substack{p \leq x \\ p \text{ splits in } K\\ p \text{ inert in } K'}}\mathbf{1}_{I}(\theta_{E}(p)) & = \sum_{\substack{p \leq x \\ p \text{ splits in } K}}\frac{1-\chi_{K'}(p)}{2}\mathbf{1}_{I}(\theta_{E}(p)).
\end{align*} 
By the proof of Theorem \ref{master} (2.b), we have that there exists an absolute constant $\Cl[abcon]{475}$ for which the above equals \begin{align*}
    \frac{|I|}{2\pi}\bigg(\sum_{\substack{p \leq x \\ p \text{ splits in } K}}1\bigg)+O\bigg(\frac{x}{\log x}\exp \bigg ( \frac{-\Cr{475}\log x}{\sqrt{\log x} + \log N_E} \bigg ) \bigg).
\end{align*} Here the Siegel zero contribution is omitted as there are only finitely many possible moduli for $\chi_K$.

Now, term 3 can be dealt with in the same way as term 2. Term 4 can be estimated using \eqref{arithprogestimate}. Collating the above results, we have that when $K\not\sim K'$, there exists an absolute constant $\Cl[abcon]{476}$ such that \begin{align*}
    \pi_{I, I'}(x)&=\frac{|I||I'|}{\pi^2}\bigg(\sum_{\substack{p \leq x \\ p \text{ splits in } K\\ p \text{ splits in } K'}}1\bigg)+\frac{|I|\mathbf{1}_{\frac{\pi}{2}\in I'}}{2\pi}\bigg(\sum_{\substack{p \leq x \\ p \text{ splits in } K}}1\bigg) \\& + \frac{|I'|\mathbf{1}_{\frac{\pi}{2}\in I}}{2\pi}\bigg(\sum_{\substack{p \leq x \\ p \text{ splits in } K'}}1\bigg)+\mathbf{1}_{\frac{\pi}{2}\in I}\mathbf{1}_{\frac{\pi}{2}\in I'}\bigg(\sum_{\substack{p \leq x \\ p \text{ inert in } K\\ p \text{ inert in }K'}}1\bigg)\\ & +O\bigg(x(\log x)^4\exp\bigg(\frac{-\Cr{474}\log x}{\sqrt{\log x}+\log N_{E}N_{E'}}\bigg)\bigg)\\& = \mu_{\text{ST}}^1(I)\mu_{\text{ST}}^2(I')\pi(x)+O\bigg(x(\log x)^4\exp\bigg(\frac{-\Cr{476}\log x}{\sqrt{\log x}+\log N_{E}N_{E'}}\bigg)\bigg).
\end{align*} 
When $K\sim K'$, there exists an absolute constant $\Cl[abcon]{477}$ such that \begin{align*}
    \pi_{I, I'}(x)&=\frac{|I||I'|}{\pi^2}\bigg(\sum_{\substack{p \leq x \\ p \text{ splits in } K\sim K'}}1\bigg)+\mathbf{1}_{\frac{\pi}{2}\in I}\mathbf{1}_{\frac{\pi}{2}\in I'}\bigg(\sum_{\substack{p \leq x \\ p \text{ inert in } K\sim K'}}1\bigg)\\ & +O\bigg(x(\log x)^4\exp\bigg(\frac{-\Cr{474}\log x}{\sqrt{\log x}+\log N_{E}N_{E'}}\bigg)\bigg)\\& = \frac{1}{2}\bigg(\frac{|I||I'|}{\pi^2}+\mathbf{1}_{\frac{\pi}{2}\in I}\mathbf{1}_{\frac{\pi}{2}\in I'}\bigg)\pi(x)+O\bigg(x(\log x)^4\exp\bigg(\frac{-\Cr{477}\log x}{\sqrt{\log x}+\log N_{E}N_{E'}}\bigg)\bigg).
\end{align*}
Hence both the desired bounds have been proved.
\end{proof}

\newpage

\appendix

\begin{center}
\section{The Sato--Tate distributions of Double Quadric and K3 Surfaces}\label{figure}
\end{center}
\bigskip

\begin{figure}[h]
    \captionsetup{justification=centering}
    {\centering
    \includegraphics[scale=0.5]{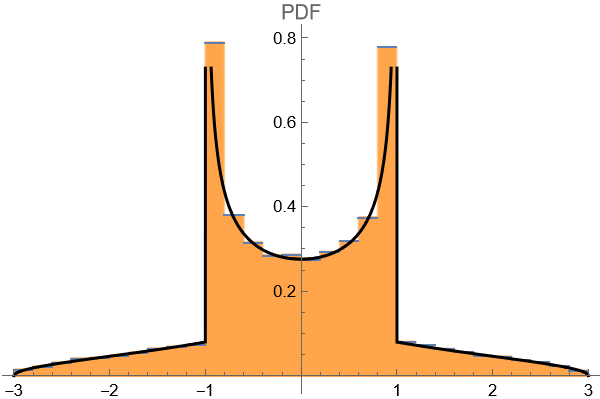}}
    
    \caption{Graph of $a_{X_\lambda}^\ast$, $\lambda = 2, p \leq 5\times10^5$. We overlay this with the ``Batman'' distribution.}
    \label{batman}
\end{figure}

\begin{figure}[h]
    \captionsetup{justification=centering}
    {\centering
    \includegraphics[scale=0.5]{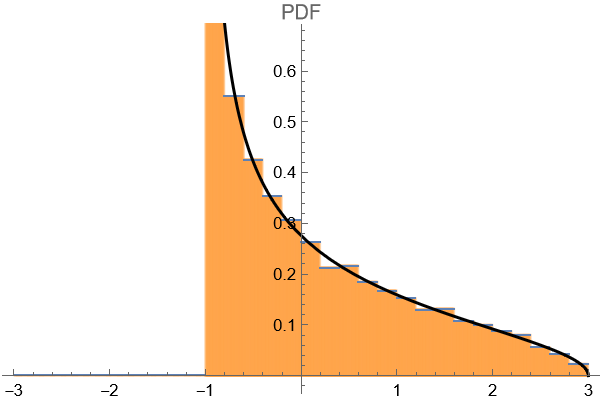}}
    
    \caption{Graph of $a_{X_\lambda}^\ast$, $\lambda = 3, p \leq 2\times10^5$. We overlay this with the graph of $\frac{1}{2\pi}\sqrt{\frac{3-x}{1+x}}$.}
    \label{halfbatman}
\end{figure}

\begin{figure}[h]
    \captionsetup{justification=centering}
    {\centering
    \includegraphics[scale=0.5]{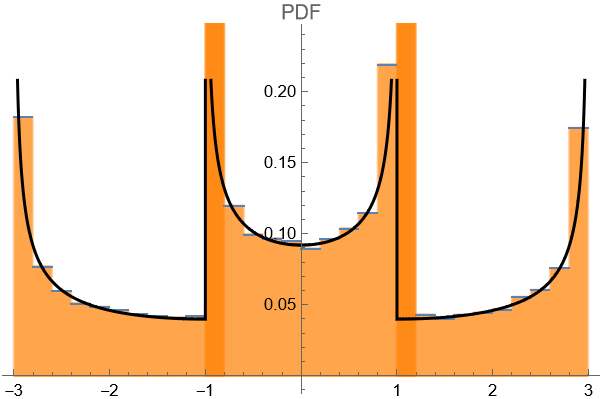}}
    
    \caption{Graph of $a_{X_\lambda}^\ast$, $\lambda = 1, p \leq 5\times10^5$. We overlay this with the ``flying Batman'' distribution.}
    \label{flyingbatman}
\end{figure}

\newpage

\begin{figure}[h]
    \captionsetup{justification=centering}
    {\centering
    \includegraphics[scale=0.5]{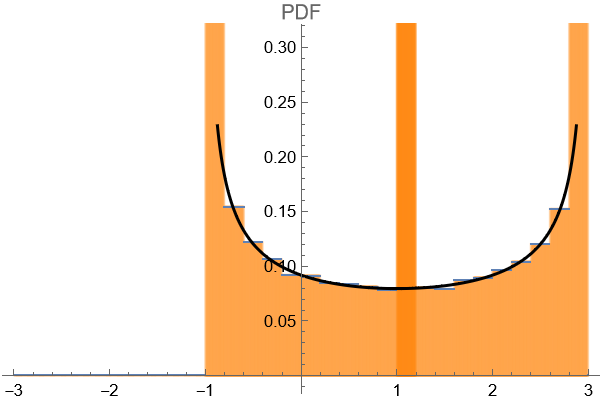}}
    
    \caption{Graph of $a_{X_\lambda}^\ast$, $\lambda = -4, p \leq 5\times10^5$. We overlay this with the graph of $\frac{1}{2\pi \sqrt{3 + 2x - x^2}}$.}
    \label{lambdanegative4}
\end{figure}

\begin{figure}[h]
    \captionsetup{justification=centering}
    {\centering
    \includegraphics[scale=0.5]{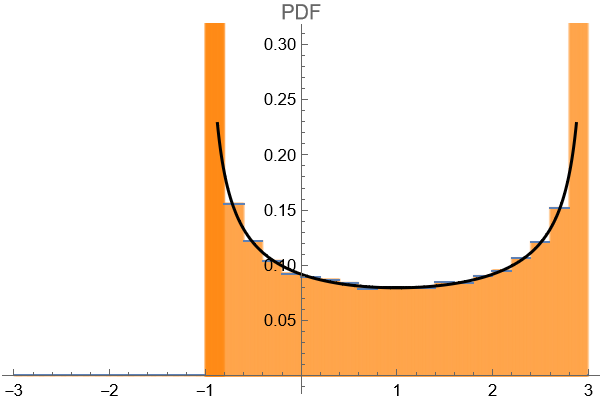}}
    
    \caption{Graph of $a_{X_\lambda}^\ast$, $\lambda = 8, p \leq 5\times10^5$. We overlay this with the graph of $\frac{1}{2\pi \sqrt{3 + 2x - x^2}}$.}
    \label{lambda8}
\end{figure}

\begin{figure}[h]
    \captionsetup{justification=centering}
    {\centering
    \includegraphics[scale=0.5]{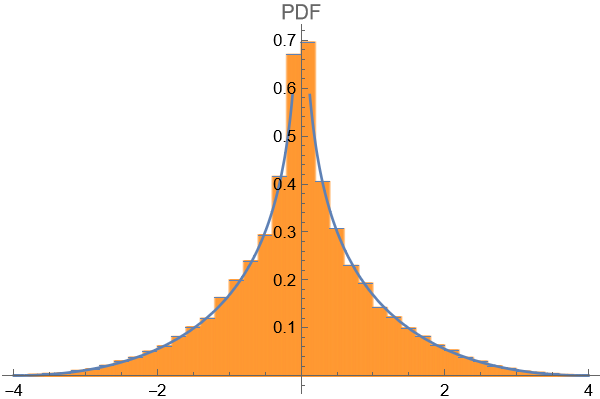}}
    
    \caption{Graph of $a_{\mathcal{Z}}^\ast$, where the two elliptic curves are $x_1^3+7x_1z_1^2+13z_1^3$ and $x_1^3+7x_1z_1^2+17z_1^3$, and $p \leq 5\times10^5$. We overlay this with the graph of $C_1$.}
    \label{cone}
\end{figure}

\newpage

\begin{figure}[h]
    \captionsetup{justification=centering}
    {\centering
    \includegraphics[scale=0.5]{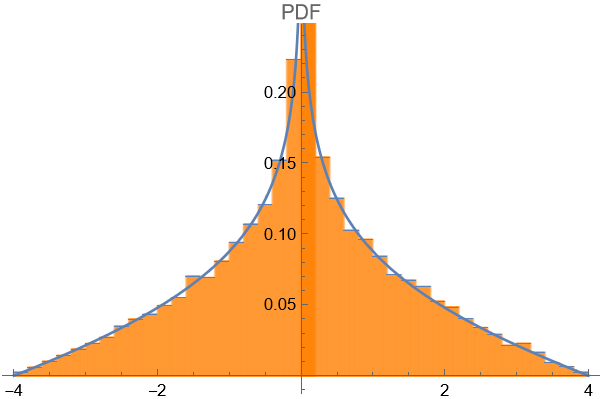}}
    
    \caption{Graph of $a_{\mathcal{Z}}^\ast$, where the two elliptic curves are $x_1^3+7x_1z_1^2+13z_1^3$ and $x_1^3-99x_1z_1^2+378z_1^3$, and $p \leq 5\times10^5$. We overlay this with the graph of $C_2$.}
    \label{ctwo}
\end{figure}

\begin{figure}[h]
    \captionsetup{justification=centering}
    {\centering
    \includegraphics[scale=0.5]{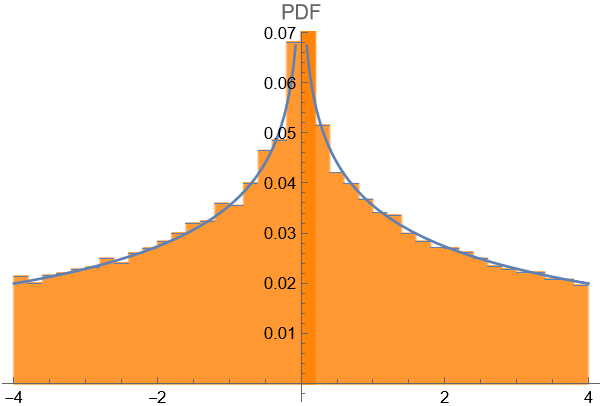}}
    
    \caption{Graph of $a_{\mathcal{Z}}^\ast$, where the two elliptic curves are $x_1^3-11x_1z_1^2-14z_1^3$ and $x_1^3-120x_1z_1^2+506z_1^3$, and $p \leq 2\times10^6$. We overlay this with the graph of $C_3$.}
    \label{cthree}
\end{figure}

\newpage

\newcommand{\etalchar}[1]{$^{#1}$}

%\bibliographystyle{alpha}
%\bibliography{main}

\end{document}